\documentclass[preprint,10pt]{elsarticle}
\usepackage{lineno,hyperref}
\usepackage{graphicx}
\usepackage[leqno]{amsmath}\allowdisplaybreaks[1]
\usepackage{amssymb,amsthm} 
\usepackage{epstopdf}
\usepackage{epsfig}
\usepackage{multirow}
\usepackage{booktabs}

\newtheorem{thm}{\bf Theorem}

\newtheorem{lemma}{\bf Lemma}

\newtheorem{proposition}{\bf Proposition}
\newtheorem{example}{\bf Example}

\newenvironment{sequation*}{\small
\begin{equation*}}{\end{equation*}}

\begin{document}

\begin{frontmatter}

\title{Estimation and Model Identification of Locally Stationary Varying-Coefficient Additive Models}

\author{Lixia Hu, Tao Huang, Jinhong You}
\address{Shanghai University of Finance and Economics}




\begin{abstract}
Nonparametric regression models with locally stationary covariates have received increasing interest
in recent years. As a nice relief of  ``curse of dimensionality'' induced by large dimension of covariates, additive regression model is  commonly used. However, in locally stationary context, to catch the dynamic nature of regression function, we adopt a flexible varying-coefficient additive model where the regression function has the form
$\alpha_{0}\left(u\right)+\sum_{k=1}^{p}\alpha_{k}\left(u\right)\beta_{k}\left(x_{k}\right).$
For this model, we propose a three-step spline estimation method for each univariate nonparametric function,
and show its consistency and $L_{2}$ rate of convergence.
Furthermore, based upon the three-step estimators, we develop a two-stage penalty procedure to identify pure additive terms
and varying-coefficient terms in varying-coefficient additive model.
As expected, we demonstrate that the proposed identification procedure is consistent, and the penalized estimators achieve the same $L_{2}$ rate of convergence as the polynomial spline estimators.
Simulation studies are presented to illustrate the finite sample performance of the proposed three-step spline estimation method and two-stage model selection procedure.

\end{abstract}


\end{frontmatter}


\section{Introduction}
Modelling nonparametric time series has received increasing interest among scholars for a few decades,
see, for example, \cite{cai2000functional,fan1998direct,fan2002,hastie1990generalized,kim1999}.
In classical time series analysis, the stationarity of time series is a fundamental assumption.
Yet,  it may be violated on some occasions in such the fields as finance, sound analysis and neuroscience,
especially when the time span of observations tends to infinity.
So, it is necessary to generalize the stationary process to the nonstationary process.
Priestley (1965) \cite{Priestley1965} first introduced a stochastic process with evolutionary spectra,
which locally displays an approximately stationary behavior.
But in his framework, it is impossible to establish an asymptotic statistical inference.
Dahlhaus (1997) proposed a new generalization  of stationarity, called locally stationary process,
and investigated its statistical inference. More details can refer to
\cite{dahlhaus1996asymptotic,dahlhaus1996kullback,dahlhaus1997fitting}.
In essence, the  locally stationary process is  locally close  to a stationary process over short periods of time,
but its second order characteristic is gradually changing as time evolves.
A formal description of locally stationary process can refer to Assumption (A1) in the Appendix.
In parametric context, the statistical inference of locally stationary process has been studied extensively by
\cite{dahlhaus1999nonlinear,dahlhaus2006statistical,fryzlewicz2008normalized,hafner2010efficient,koo2010semiparametric}.
In nonparametric context, Vogt (2012) \cite{vogt2012nonparametric} considered the time-varying nonlinear autoregressive (tvNAR) models
including its general form and estimated the time-varying multivariate regression function  using the kernel-type method.
 However,  it still suffers the ``curse of dimensionality'' problem when the dimension of covariates is high.

In order to solve the aforementioned problem,  a familiar way is to adopt the additive nonparametric regression model
suggested by \cite{hastie1990generalized}.
It not only remedies  the  ``curse of dimensionality'', but also has an independent interest  in  practical applications
due to its flexibility and interpretability.
There exists abound research findings about the additive regression model in the literature.
In the case of iid observations, the additive nonparametric component functions can be estimated by kernel-based methods:
the classic backfitting estimators of \cite{hastie1990generalized}, the marginal integration estimators of
\cite{fan1998direct,linton1995kernel}, the smoothing backfitting estimators of \cite{mammen1999existence}, and two-stage estimators of \cite{linton1997efficient,horowitz2006optimal}.
In the stationary time series context, there are  kernel estimators via marginal integration of \cite{tjostheim1994nonparametric,yang1999nonparametric}, spline estimators of \cite{huang1998projection,huang2004identification,stone1985,stone1994,xue2006estimation},
and the spline-backfitted kernel (SBK) estimators which borrow the strength of both kernel estimation
and spline estimation, see \cite{liu2013oracle,wang2007spline,wang2009efficient}.
Vogt \cite{vogt2012nonparametric} considered the  locally stationary additive model and proposed smooth backfitting method
 to estimate bivariate additive component functions.

On the other hand, the varying-coefficient model is a natural extension of linear model which allows the coefficients to change over certain common covariates instead of being invariant.
This model succeeds to relax the parameter limitation of linear model and may have practical as well as theoretical significance.
For this model, there are three types of estimation methods: local polynomial smoothing method \cite{fan1999,wu1998},
polynomial spline estimation method \cite{huang2002,huang2004functional,huang2004polynomial} and smoothing spline method
\cite{chiang2001,hastie1993vary,hoover1998}.

Zhang and Wang \cite{zhang2015functional} proposed a so-called varying-coefficient additive model
 to catch the evolutionary nature of time-varying regression function in the analysis of functional data.
Their model assumes the evolutionary regression function has the form
$m\left(t,\mathbf{x}\right)=\alpha_{0}\left(t\right)+\sum_{k=1}^{p}\alpha_{k}\left(t\right)\beta_{k}\left(x_{k}\right),$
which is more flexible in the sense that it covers both varying-coefficient model and additive model as special cases.
Specifically speaking, it reduces to an additive model when $\alpha_{k},k=1,\cdots,p,$ are all constants, and a varying-coefficient
model if $\beta_{k},k=1,\cdots,p,$ are all linear functions. Extracting the special meaning of time in functional data analysis,
one can generalize time to some other common covariates.

In this paper,  we model locally stationary time series. To  concreteness,
let $\big\{Y_{t,T},X_{t,T}^{\left(1\right)},\cdots,X_{t,T}^{\left(p\right)}\big\}_{t=1}^{T}$
 be a length-$T$ realization of $p+1$ dimension locally stationary time series, and  assume that the data is  generated by
 varying-coefficient additive model as follows
\begin{equation}\label{regression}
Y_{t,T}=\alpha_{0}\left(\frac{t}{T}\right)+\sum_{k=1}^{p}\alpha_{k}\left(\frac{t}{T}\right)\beta_{k}\left(X_{t,T}^{\left(k\right)}\right)
+\sigma\left(\frac{t}{T},\mathbf{X}_{t,T}\right)\varepsilon_{t},\ \ t=1,\cdots,T,
\end{equation}
where $\varepsilon_t$'s are i.i.d,
$\alpha_{k}$ is the varying-coefficient component function, $\beta_{k}$ is the additive component  function
and $\sigma$ is a bivariate nonparametric function, which allows the heteroscedasticity case.
Without loss of generality, we require
$\textstyle{\mathbf{X}_{t,T}=\big\{X_{t,T}^{\left(1\right)},\cdots,X_{t,T}^{\left(p\right)}\big\}^{\tau}\in[0,1]^{p}},$
where the superscript `$\tau$' means transposition of vector or matrix.
In order to identify these multiplied component functions,  we require that
\[\parallel \alpha_{1}\parallel_{L_{2}}=\cdots=\parallel\alpha_{p}\parallel_{L_{2}}=1\ \text{and}\  \beta_{1}\left(0\right)=\cdots\beta_{p}\left(0\right)=0,\]
where $\parallel f\parallel_{L_{2}}=\{\int_{0}^{1}f^{2}\left(z\right)\mathrm{d}z\}^{1/2}$ is the $L_{2}$ norm of any function $f$ defined on $[0,1]$
such that $\int_{0}^{1}f^{2}\left(z\right)\mathrm{d}z<\infty.$

For functional data, Zhang and Wang \cite{zhang2015functional} proposed a two-step spline estimation procedure.
In the first step, sorting the data within each subject in ascending order of time
and averaging the response for each subject using trapezoidal rule to fit an additive model, then, in the second step,
 fitting a varying-coefficient model by substituting the estimated additive function into varying-coefficient additive model.
His estimation methodology  works since there are dense observation for every subject and covariates is independent of observation time within subject.

However, for some other practical problems, such as longitudinal data with finite observertion time, time series data, such an assumption fails.
To circumvent  this problem, under mild assumptions, we derive an initial estimation of additive component by employing a   segmentation technique.
Then we can fit a varying-coefficient model and an additive model, respectively, to get spline estimators of varying-coefficient function and
additive function.
As expected, we show that the proposed estimators of $\alpha_{k}$ and $\beta_{k}$  are consistent
 and present the corresponding $L_{2}$  rate of convergence.

On the other hand, the product term in \eqref{regression}  may simply reduce to a varying-coefficient term  or an
additive term in the case of $\beta_{k}$ being linear function or $\alpha_{k}$ being constant.
So, in the parsimony sense, identifying additive terms and varying-coefficient terms in \eqref{regression} are of interest.
To this end, we propose a two-stage penalized least squares estimator based on SCAD penalty function, and, furthermore,  show that our model identification strategy is consistent, i.e., the  additive term and the varying-coefficient  term
are correctly selected with probability approaching to 1. Meantime, $L_{2}$ rate of convergence of penalized spline estimator of each component function achieves the rate of the spline estimator of univariate nonparametric function.

The rest of this paper is organized as follows. We propose a three-step spline estimation method in Section 2 and
 a two-stage model identification procedure in Section 3.
 Section 4 describes the smoothing parameter selection strategies.
 Section 5 establishes the asymptotic properties of the proposed model estimation and identification methods.
Simulation studies   are illustrated in Section 6.
The main technical proofs are presented in the Appendix.
Lemmas and other similar proofs are given in the Supplementary.

\section{Model Estimation}

In this section, we propose a three-step spline estimation method for the proposed locally stationary varying-coefficient additive model \eqref{regression}.

\begin{itemize}
\item{Step I}:  Segment the rescaled time $\left(\frac{1}{T},\frac{2}{T},\cdots,1\right)$ into several groups, and approximate each varying-coefficient function $\alpha_{k}$ within the same group by a local constant.  Thus, model \eqref{regression} can be approximated by an additive model, and a scaled-version of the additive component functions $\beta_{k}$ can be obtained using spline-based method.
\item{Step II}: Substitute the initial estimates of the scaled additive component functions into model \eqref{regression} to yield an approximated varying-coefficient model, and then obtain spline estimators of varying-coefficient component functions $\alpha_{k}.$
\item{Step III}: Plug-in spline estimators of varying-coefficient component functions $\alpha_{k}$ into model \eqref{regression} to yield an approximated additive model, and then update the spline estimation of additive component functions $\beta_{k}$.
\end{itemize}

We first present with some notations before detailing our proposed estimation method.
Let $\{B_{kl,A},l=1,\cdots,J_{k,A}\}$ be $p_{k}$ order B-spline basis with $K_{k,A}$ interior knots
and $J_{k,A}=K_{k,A}+p_{k}$ is the number of B-spline functions estimating additive component function $\beta_{k}.$
Similarly, we denote  $\{B_{kl,C},l=1,\cdots,J_{k,C}\}$ as $q_{k}$ order B-spline basis with $K_{k,C}$ interior knots,
and $J_{k,C}=K_{k,C}+q_{k}$  is the number of B-spline functions estimating varying-coefficient component function $\alpha_{k}.$
Here,  `$A$' and `$C$' in the subscript of B-spline functions and knots number mean that is for the additive component function
and varying-coefficient function, respectively.
Denote $\psi_{kl}=J_{k,A}^{1/2}B_{kl,A}$ and  $\varphi_{kl}=J_{k,C}^{1/2}B_{kl,C}.$  The nice properties of scaled B-spline basis
are listed in the Appendix. \\

{\bf {Step I: Initial estimators of scaled additive component functions}}

We segment the sample $\{\left(Y_{t,T},\mathbf{X}_{t,T}\right)^{\mathbf{\tau}},t=1,\cdots,T\}$ in ascending order of time into $N_{T}$ groups with $I_{T}$ observations in each group, where $\textstyle{I_{T}}$ hinges on the sample size $T$ and $I_{T}\times N_{T}=T$.
Then approximate  $\alpha_{k}(u)$ in the $s$th group, i.e. $\left(s-1\right)I_{T}+1 \leq uT \leq sI_{T},$ by a constant $C_{ks}$,  where $C_{ks}$ is some constant dependent on $k$ and $s$ such that $\sum_{s=1}^{N_{T}}C_{ks}/N_{T}<\infty.$

For the sake of convenient presentation, we suppress the triangular array index in locally stationary time series,
and represent time index in the $s$th  group  as $t_{s1},\cdots,t_{sI_{T}}$ for given $s=1,\cdots N_{T}.$
 Then one can approximate model \eqref{regression} as
\begin{equation}\label{subregression1}
Y_{t_{sj},T}\approx C_{0s}+\sum_{k=1}^p C_{ks}\beta_{k}(X_{t_{sj},T}^{\left(k\right)})
+\sigma_{t_{sj}}\varepsilon_{t_{sj}},\ s=1,\cdots,N_{T}, \ j=1,\cdots,I_{T}.
\end{equation}
where $\sigma_{t_{sj}}=\sigma\left(t_{sj}/T,\mathbf{X}_{t_{sj},T}\right).$
If $C_{ks},k=0,\cdots,p,s=1,\cdots,N_{T}$ are all known, one can easily construct the spline estimator of $\beta_{k}.$
Suppose that $(\hat{C}_{0s},\hat{h}_{kl},l=1,\cdots,J_{k,A},k=1,\cdots,p,s=1,\cdots,N_{T})^{\tau}$ minimizes
\begin{equation}\label{minimize2}
\sum_{s=1}^{N_{T}}\sum_{j=1}^{I_{T}}\big[Y_{t_{sj},T}-C_{0s}
-\sum_{k=1}^{p}\sum_{l=1}^{J_{k,A}}C_{ks}h_{kl}\psi_{kl}(X_{t_{sj},T}^{\left(k\right)})\big]^2,
\end{equation}
then $\hat{\beta}_{k}\left(x_{k}\right)=\sum_{l=1}^{J_{k,A}}\hat{h}_{kl}\bar{\psi}_{kl}\left(x_{k}\right),$
where $\bar{\psi}_{kl}\left(\cdot\right)=\psi_{kl}\left(\cdot\right)-\psi_{kl}\left(0\right).$

However, $C_{ks}$'s are unknown. We instead rewrite \eqref{subregression1} as an additive model,
\begin{equation}\label{subregression2}
Y_{t_{sj},T}\approx C_{0s}+\sum_{k=1}^p\beta_{k}^{\left(s\right)}(X_{t_{sj},T}^{\left(k\right)})
+\sigma_{t_{sj}}\varepsilon_{t_{sj}},\ s=1,\cdots,N_{T},\ j=1,\cdots,I_{T},
\end{equation}
where $\beta_{k}^{\left(s\right)}\left(\cdot\right)=C_{ks}\beta_{k}\left(\cdot\right).$
For each given $s,$ let $(\tilde{C}_{0s},\hat{h}_{kl}^{\left(s\right)},l=1,\cdots,J_{k,A},k=1,\cdots,p)^{\tau}$
minimize
\begin{equation}\label{minimize1}
\sum_{t=1}^{I_{T}}\big[Y_{t_{sj},T}-C_{0s}-\sum_{k=1}^{p}\sum_{l=1}^{J_{k,A}}
h_{kl}^{\left(s\right)}\psi_{kl}(X_{t_{sj},T}^{\left(k\right)})\big]^2.
\end{equation}
By \eqref{minimize2} and \eqref{minimize1}, it is easy to see that
\begin{equation*}\label{compare}
\hat{C}_{0s}=\tilde{C}_{0s}, \ \ \hat{h}_{kl}^{\left(s\right)}=C_{ks}\hat{h}_{kl},\ l=1,\cdots,J_{k,A},\ k=1,\cdots,p,\ s=1,\cdots,N_{T},
\end{equation*}
which implies
\begin{eqnarray}\label{relationsol}
\frac{1}{N_{T}}\sum_{s=1}^{N_{T}}C_{ks}\hat{\beta}_{k}\left(x_{k}\right)&=&
\frac{1}{N_{T}}\sum_{s=1}^{N_{T}}C_{ks}\sum_{l=1}^{J_{k,A}}\hat{h}_{kl}\bar{\psi}_{kl}\left(x_{k}\right) \\
&=&\frac{1}{N_{T}}\sum_{s=1}^{N_{T}}\sum_{l=1}^{J_{k,A}}\hat{h}_{kl}^{\left(s\right)}\bar{\psi}_{kl}\left(x_{k}\right)
\equiv \hat{\gamma}_{k}\left(x_{k} \right). \nonumber
\end{eqnarray}
In a word, although the additive component function in \eqref{subregression1} cannot be estimated directly,
the scaled additive component function
$ \gamma_{k}\left(\cdot\right)=w_{k}\beta_{k}\left(\cdot\right)$
with $w_{k}=\sum_{s=1}^{N_{T}}C_{ks}/N_T$
is estimable using the proposed segmentation techniques. \\

{\bf {Step II: Spline estimators of varying-coefficient component functions}}

Define $\delta_{0}\left(\cdot\right)=\alpha_{0}\left(\cdot\right)$ and $\delta_{k}\left(\cdot\right)=\alpha_{k}\left(\cdot\right)/w_{k}$, $k=1,\cdots,p$.
By \eqref{relationsol},  substituting $\hat{\gamma}_{k}\left(\cdot\right)$ into \eqref{regression} yields
\begin{equation}\label{stageI}
Y_{t,T}=\delta_{0}\left(t/T\right)+\sum_{k=1}^{p}\delta_{k}\left(t/T\right)\hat{\gamma}_{k}(X_{t,T}^{\left(k\right)})+e_{t,T},\ \ t=1,\cdots,T,
\end{equation}
where $e_{t,T}=\sigma(t/T,\mathbf{X}_{t,T})\varepsilon_{t}-\sum_{k=1}^{p}\delta_{k}(t/T)
\big[\hat{\gamma}_{k}(X_{t,T}^{\left(k\right)})-\gamma_{k}(X_{t,T}^{\left(k\right)})\big].$

Model \eqref{stageI} can be viewed as  a varying-coefficient model, and the spline estimators of varying-coefficient functions $\delta_{k}\left(\cdot\right),k=0,\cdots,p$ are easily obtained. Suppose that
$\left(\hat{g}_{kl},l=1,\cdots,J_{k,C},k=0,\cdots,p\right)^{\tau}$ minimizes
\[\sum_{t=1}^{T}\big[Y_{t,T}-\sum_{l=1}^{J_{0,C}}g_{0l}\varphi_{0l}\left(t/T\right)
-\sum_{k=1}^{p}\hat{\gamma}_{k}(X_{t,T}^{\left(k\right)})\sum_{l=1}^{J_{k,C}}g_{kl}\varphi_{kl}\left(t/T\right)\big]^2.\]
Then,
\[\hat{\delta}_{0}\left(\cdot\right)=\sum_{l=1}^{J_{0,C}}\hat{g}_{0l}\varphi_{0l}\left(\cdot\right),\quad
\hat{\delta}_{k}\left(\cdot\right)=\sum_{l=1}^{J_{k,C}}\hat{g}_{kl}\varphi_{kl}\left(\cdot\right),\ k=1,\cdots,p.\]
By the definition of $\delta_{k}\left(\cdot\right)$ and identifiability conditions for $\alpha_{k}\left(\cdot\right),k=1,\cdots,p,$  we have the
spline estimators of varying-coefficient functions $\alpha_{k}\left(\cdot\right)$'s in model \eqref{regression} as
\begin{equation}\label{hatalpha}
\hat{\alpha}_{0}\left(\cdot\right)=\hat{\delta}_{0}\left(\cdot\right),\ \
\hat{\alpha}_{k}\left(\cdot\right)=  \hat{\delta}_{k}\left(\cdot\right)/
\parallel \hat{\delta}_{k}\left(\cdot\right)\parallel,\ \  k=1,\cdots,p.
\end{equation}

{\bf {Step III: Spline estimators of additive component functions}}

Substituting \eqref{hatalpha} into model \eqref{regression} yields
\begin{equation}\label{stageIII}
Y_{t,T}=\hat{\alpha}_{0}\left(t/T\right)+\sum_{k=1}^{p}\hat{\alpha}_{k}\left(t/T\right)\beta_{k}(X_{t,T}^{\left(k\right)})+\eta_{t,T},
\end{equation}
where
$\eta_{t,T}=\sigma(\frac{t}{T},\mathbf{X}_{t,T})\varepsilon_{t}+\alpha_{0}(\frac{t}{T})-\hat{\alpha}_{0}(\frac{t}{T})
+\sum_{k=1}^{p}\big[\alpha_{k}(\frac{t}{T})-\hat{\alpha}_{k}(\frac{t}{T})\big]\beta_{k}(X_{t,T}^{\left(k\right)}).$

Model \eqref{stageIII} can be viewed as  a varying-coefficient model. Suppose that $(\hat{f}_{kl},l=1,\cdots,J_{k,A},k=1,\cdots,p)^{\tau}$ minimizes
\[\sum_{t=1}^{T}\big[Y_{t,T}-\hat{\alpha}_{0}(t/T)
-\sum_{k=1}^{p}\hat{\alpha}_{k}(t/T)\sum_{l=1}^{J_{k,A}}f_{kl}\psi_{kl}(X_{t,T}^{\left(k\right)})\big]^2,\]
Then, spline estimators of additive component functions in \eqref{regression} are given by
$$\hat{\beta}_{k}\left(\cdot\right)=\sum_{l=1}^{J_{k,A}}\hat{f}_{kl}\bar{\psi}_{kl}\left(\cdot\right), k=1,\cdots,p.$$

{\bf Remark 1:}
The spline estimators  $\hat{\alpha}_{k}$ and $\hat{\beta}_{k}$ can be updated by iterating Step II and Step III.
However, one step estimation is enough and there is no great improvement through iteration procedure.

\vskip 2in 

{\bf Remark 2:}
One may employ different B-spline basis functions in Step I  and Step III for estimating the additive component functions ${\beta}_{k}$. Yet, we don't distinguish them in symbols for the sake of simplicity.


\section{Model Identification}
The proposed varying-coefficient additive model is more general and flexible than either varying-coefficient model
or additive model, and covers them as special cases. But, in practice, a parsimonious model is always one's preference when there exist several potential options.  Hence, it is of great interest to explore whether the varying-coefficient component function is truly varying and whether the additive component function degenerates to simply linear function.  In this paper, we decompose varying-coefficient additive terms into additive terms $\beta_{k}$  and varying-coefficient terms $\alpha_{k}$, and, motivated by \cite{fan2001}, propose a two-stage penalized least squares (PLS) model identification procedure to identify the term that $\beta_{k}$ is constant ($\beta_{k}''=0$) or/and $\alpha_{k}$   is linear ($\alpha_{k}'=0$).

\begin{itemize}
\item{Stage I}: Plug-in the spline estimators of additive component functions $\beta_{k}$ obtained in the estimation stage into model \eqref{regression}, and penalize $\parallel\alpha_{k}'\parallel_{L_{2}}$ to identify linear additive terms.
\item{Stage II}: Given the penalized spline estimators  of additive component functions $\beta_{k}$ obtained in Stage I of the model identification process, penalize $\parallel\beta_{k}''\parallel_{L_{2}}$ to select constant varying-coefficient terms.
\end{itemize}

We first introduce some notations. Let  $K_{A}=\max_{1\leq k\leq p }K_{k,A}$
and $K_{C}=\max_{0\leq k\leq p }K_{k,C}.$ Denote
 \begin{equation*}
 \begin{aligned}
 \Phi_{k}(\cdot)=&\{\varphi_{k1}(\cdot),\cdots,\varphi_{kJ_{k,C}}(\cdot)\}^{\tau},&k=0,\cdots,p,\\
\Psi_{k}\left(\cdot\right)=&\{\psi_{k1}\left(\cdot\right),\cdots,\psi_{kJ_{k,A}}\left(\cdot\right)\}^{\tau}, &  k=1,\cdots,p.\\
\end{aligned}
\end{equation*}
and
 \begin{equation*}
 \begin{aligned}
W_{k}=&\Big\{\int_{0}^{1}\Phi'_{kj}\left(u\right)\Phi'_{kj'}\left(u\right)\mathrm{d}u\Big\}_{J_{k,C}\times J_{k,C}},&k=0,\cdots,p,\\
V_{k}=&\Big\{\int_{0}^{1}\Psi''_{kj}\left(u\right)\Psi''_{kj'}\left(u\right)\mathrm{d}u\Big\}_{J_{k,A}\times J_{k,A}},&  k=1,\cdots,p.
\end{aligned}
\end{equation*}

{\bf {Stage I: Identifying linear additive terms}}

By substituting the additive component functions  ${\beta}_{k}(\cdot)$ by their spline estimates $\hat{\beta}_{k}(\cdot)$ obtained in the estimation stage, model
\eqref{regression} becomes
\begin{equation*}
Y_{t,T}=\alpha_{0}\left(t/T\right)+\sum_{k=1}^{p}\alpha_{k}\left(t/T\right)\hat{\beta}_{k}\left(X_{t,T}^{\left(k\right)}\right)
+\tilde{\varepsilon}_{t},
\end{equation*}
where $\tilde{\varepsilon}_{t}= \sigma(t/T,\mathbf{X}_{t,T})\varepsilon_{t}+\sum_{k=1}^{p}\alpha_{k}(t/T)
\big[\beta_{k}(X_{t,T}^{\left(k\right)})-\hat{\beta}_{k}(X_{t,T}^{\left(k\right)})\big].$

Let
$\mathbf{\pi}=\left(\pi_{0}^{\tau},\cdots,\pi_{p}^{\tau}\right)^{\tau}$ with $\pi_{k}=\left(\pi_{k1},\cdots,\pi_{kJ_{k,C}}\right)^{\tau}$, and assume
$\hat{\pi}=\left(\hat{\pi}_{0},\cdots,\hat{\pi}_{p}\right)^{\tau}$ is determined by
\begin{equation}\label{SCAD1}
\begin{aligned}
\hat{\mathbf{\pi}}
=&\operatorname*{argmin}_{\mathbf{\pi}}\frac{1}{2}\sum_{t=1}^{T}\Big[Y_{t,T}-\pi_{0}^{\tau}\Phi_{0}\left(t/T\right)
-\sum_{k=1}^{p}\pi_{k}^{\tau}\Phi_{k}\left(t/T\right)\hat{\beta}_{k}\left(X_{t,T}^{\left(k\right)}\right)\Big]^2\\
&+T\sum_{k=1}^{p}p_{\lambda_{T}}\left(K_{C}^{-3/2}\parallel\alpha'_{k}\parallel_{L_{2}}\right),
\end{aligned}
\end{equation}
where $\parallel\alpha'_{k}\parallel_{L_{2}}^{2}=\parallel\pi_{k}^{\tau}\Phi_{k}'\parallel_{L_{2}}=\pi_{k}^{\tau}W_{k}\pi_{k},$
and $p_{\lambda_{T}}\left(\cdot\right)$ is a penalty function with a tuning parameter $\lambda_{T}.$
Then the penalized spline estimators of $\alpha_{k}\left(\cdot\right)$ are given by
\[\hat{\alpha}_{0}^{P}\left(\cdot\right)=\hat{\pi}_{0}^{\tau}\Phi_{0}\left(\cdot\right),\
\hat{\alpha}_{k}^{P}\left(\cdot\right)=\hat{\pi}_{k}^{\tau}\Phi_{k}\left(\cdot\right)/
\parallel\hat{\pi}_{k}^{\tau}\Phi_{k}\parallel_{L_{2}},\ k=1,\cdots,p.\]
Here the superscript `$P$' denotes the penalized spline estimation.  \\

{\bf {Stage II: Identifying constant varying-coefficient terms}}

By replacing the varying-coefficient function $\alpha_{k}\left(\cdot\right)$ with their penalized spline estimates $\hat{\alpha}_{k}^{P}\left(\cdot\right)$ obtained in Stage I, model \eqref{regression} becomes
\begin{equation*}
Y_{t,T}=\hat{\alpha}_{0}^{P}\left(t/T\right)+\sum_{k=1}^{p}\hat{\alpha}_{k}^{P}\left(t/T\right)
\beta_{k}\left(X_{t,T}^{\left(k\right)}\right)+\tilde{\varepsilon}_{t}^{P},
\end{equation*}
where $\tilde{\varepsilon}_{t}^{P}= \sigma\left(t/T,\mathbf{X}_{t,T}\right)\varepsilon_{t}
+\alpha_{0}\left(t/T\right)-\hat{\alpha}_{0}^{P}\left(t/T\right)+\sum_{k=1}^{p}
\big[\alpha_{k}\left(t/T\right)-\hat{\alpha}_{k}^{P}\left(t/T\right)\big]\beta_{k}\left(X_{t,T}^{\left(k\right)}\right).$

Let
$\mathbf{\varpi}=\left(\varpi_{1}^{\tau},\cdots,\varpi^{\tau}_{p}\right)^{\tau}$ with
$\varpi_{k}=\left(\varpi_{k1},\cdots,\varpi_{kJ_{k,A}}\right)^{\tau}$,
and assume
$\hat{\mathbf{\varpi}}=\left(\hat{\varpi}_{1}^{\tau},\cdots,\hat{\varpi}_{p}^{\tau}\right)^{\tau}$
is given by
\begin{equation}\label{SCAD2}
\begin{aligned}
\hat{\mathbf{\varpi}}
=&\operatorname*{argmin}_{\mathbf{\varpi}}\frac{1}{2}\sum_{t=1}^{T}\Big[Y_{t,T}-\hat{\alpha}_{0}^{P}\left(t/T\right)
-\sum_{k=1}^{p}\hat{\alpha}_{k}^{P}\left(t/T\right)\varpi_{k}^{\tau}\Psi_{k}\left(X_{t,T}^{\left(k\right)}\right)\Big]^2\\
&+T\sum_{k=1}^{p}p_{\mu_{T}}\left(K_{A}^{-3/2}\parallel\beta''_{k}\parallel_{L_{2}}\right),
\end{aligned}
\end{equation}
where $\parallel\beta''_{k}\parallel_{L_{2}}^{2}=\pi_{k}^{\tau}V_{k}\pi_{k}$
 and $p_{\mu_{T}}\left(\cdot\right)$ is a penalty function with a tuning parameter $\mu_{T}.$
Therefore,  the  penalized spline estimators of $\beta_{k}\left(\cdot\right)$ are given by
\[\hat{\beta}_{k}^{P}\left(\cdot\right)=\hat{\varpi}_{k}^{\tau}[\Psi_{k}\left(\cdot\right)-\Psi_{k}\left(0\right)],\ k=1,\cdots,p.\]


\section{Implementation Issues}

In this section, we discuss various implementation issues for the proposed model estimation and identification procedures.

\subsection{Smoothing Parameter Selection in  Estimation}
We predetermine the degree of polynomial spline.  Usual options are 0, 1 or 2, that is to choose linear, quadratic or cubic spline functions.  It is known that, when sufficient number of knots is used, the spline approximation method is quite stable.  Therefore, we suggest to use the same number of interior knots $K$ for all component functions and  $m_{i}$th  order B-spline basis functions in the $i$th step estimation to facilitate the computation. By experience, it is reasonable to choose $K=3,\cdots,8.$  In addition, in order to solve the least squares problem in each group  in Step I estimation, we require $1+\left(K+m_{1}\right)p<I_{T}.$
Under this constraint, we choose the optimal $K$ and $I_{T}$ by BIC
\[BIC\left(I_{T},K\right)=\log{\left(RSS/T\right)}+p\left(T/J_{K,1}\right)^{-1}\log{\left(T/J_{K,1}\right)},\]
where $RSS= \sum_{t=1}^{T}\big[Y_{t,T}-\hat{\alpha}_{0}\left(t/T\right)
-\sum_{k=1}^{p}\hat{\alpha}_{k}\left(t/T\right)\hat{\beta}_{k}(X_{t,T}^{\left(k\right)})\big]^2$
and $J_{K,1}=K+m_{3}$ is the number of B-spline basis functions used in Step III estimation.

\subsection{Computation in Model Identification}
Various penalty functions \cite{ fan2001,fu1998,Tibshirani1996,Yuan2006,Zou2006} can be used in practice.
We choose the SCAD penalty function proposed by \cite{fan2001}, which is defined by its first derivative
\[p_{\lambda}'\left(\theta\right)=\lambda I\left(\theta\leq\lambda\right)+\frac{\left(a\lambda-\theta\right)_{+}}{\left(a-1\right)}I\left(\theta>\lambda\right)\]
for some $a>2$ and $\theta>0,$ where symbol $\left(\cdot\right)_{+}=\max\left(\cdot,0\right).$
It is well-known that the SCAD penalty function  has nice properties such as unbiasedness, sparsity and continuous.
Meantime, it is singular at the origin, and have no continuous second order derivatives.
Yet it can be locally approximated by a quadratic function.

Specifically speaking, given an initial estimate $\pi_{k}^{\left(0\right)},$
or equivalently $\alpha_{k}^{\left(0\right)},$
if $\parallel\alpha_{k}^{\left(0\right)'}\parallel_{L_{2}}>0,$
then one can locally approximate $p_{\lambda_{T}}\left(\parallel\alpha'_{k}\parallel_{L_{2}}\right)$ by
\[
p_{\lambda_{T}}\left(\parallel\alpha_{k}^{\left(0\right)'}\parallel_{L_{2}}\right)
+\frac{1}{2}\frac{p'_{\lambda_{T}}\left(\parallel\alpha_{k}^{\left(0\right)'}\parallel_{L_{2}}\right)}{\parallel\alpha_{k}^{\left(0\right)'}\parallel_{L_{2}}}
\left(\parallel\alpha_{k}'\parallel_{L_{2}}^{2}-\parallel\alpha_{k}^{\left(0\right)'}\parallel_{L_{2}}^{2}\right).
\]
This implies that the objective function in \eqref{SCAD1}, denoted by $Q_{1}\left(\mathbf{\pi}\right),$ can be approximated, up to a constant, by
\[Q_{1}\left(\mathbf{\pi}\right)\approx\frac{1}{2}\left(\mathbf{Y}-\mathbf{D}_{S}\mathbf{\pi}\right)^{\tau}
\left(\mathbf{Y}-\mathbf{D}_{S}\mathbf{\pi}\right)
+\frac{1}{2}T\mathbf{\pi}^{\tau}\Omega_{1}\mathbf{\pi},\]
where $ \mathbf{Y}=\left(Y_{1,T},\cdots,Y_{T,T}\right)^{\tau},$ $\mathbf{D}_{S}=\left(D_{S1}^{\tau},\cdots,D_{ST}^{\tau}\right)^{\tau}$ with
$D_{St}=\{\hat{\mathbf{\delta}}^{\tau}\mathbf{\Phi}\left(t/T\right)\}^{\tau}$ and
$\mathbf{\hat{\delta}}=\big\{1,\hat{\beta}_{1}(X_{t,T}^{\left(1\right)}),\cdots,\hat{\beta}_{p}(X_{t,T}^{\left(p\right)})\big\}^{\tau}$ and
\begin{equation*}
\begin{aligned}
\Omega_{1}=&diag\left(\frac{p'_{\lambda_{T}}(K_{C}^{-3/2}\parallel\alpha_{0}^{\left(0\right)'}\parallel_{L_{2}})}{\parallel\alpha_{0}^{\left(0\right)'}\parallel_{L_{2}}}W_{0},
\cdots,\frac{p'_{\lambda_{T}}(K_{C}^{-3/2}\parallel\alpha_{p}^{\left(0\right)'}\parallel_{L_{2}})}{\parallel\alpha_{p}^{\left(0\right)'}\parallel_{L_{2}}}W_{p}\right).
\end{aligned}
\end{equation*}
Therefore, we can find the solution of \eqref{SCAD1} by iteratively computing the following ridge regression estimator
\[\hat{\mathbf{\pi}}=\big\{\mathbf{D}_{S}^{\tau}\mathbf{D}_{S}+T\Omega_{1}\big\}^{-1}\mathbf{D}_{S}^{\tau}\mathbf{Y}\]
until convergence.

In the same vein, we can iteratively solve the optimization problem \eqref{SCAD2}.
Let $\hat{ \mathbf{Y}}=\left(Y_{1,T}-\hat{\alpha}_{0}^{P}\left(1/T\right),\cdots,Y_{T,T}-\hat{\alpha}_{0}^{P}\left(T/T\right)\right)^{\tau},$ $\mathbf{Z}_{P}=\left(Z_{P1}^{\tau},\cdots,Z_{PT}^{\tau}\right)^{\tau}$ with
$Z_{Pt}=\big\{\hat{\mathbf{\eta}}\left(t/T\right)^{\tau}\mathbf{\Psi}\left(\mathbf{X}_{t,T}\right)\big\}^{\tau}$ and
$\mathbf{\hat{\eta}}\left(\cdot\right)=\left(\hat{\alpha}_{1}^{P}\left(\cdot\right),\cdots,\hat{\alpha}_{p}^{P}\left(\cdot\right)\right)^{\tau}$
and
\begin{equation*}
\begin{aligned}
\Omega_{2}=&diag\left(\frac{p'_{\mu_{T}}(K_{A}^{-3/2}\parallel\beta_{1}^{\left(0\right)''}\parallel_{L_{2}})}{\parallel\beta_{1}^{\left(0\right)''}\parallel_{L_{2}}}V_{1},
\cdots,\frac{p'_{\mu_{T}}(K_{A}^{-3/2}\parallel\beta_{p}^{\left(0\right)''}\parallel_{L_{2}})}{\parallel\beta_{p}^{\left(0\right)''}\parallel_{L_{2}}}V_{p}\right),
\end{aligned}
\end{equation*}
Therefore, we can iteratively compute the following ridge regression estimator
\[\hat{\mathbf{\varpi}}=\big\{\mathbf{Z}_{P}^{\tau}\mathbf{Z}_{P}+T\Omega_{2}\big\}^{-1}\mathbf{Z}_{P}^{\tau}\tilde{\mathbf{Y}}\]
until convergence.


\subsection{Tuning Parameter Selection in Model Identification}
Based on the  optimal segment length $\hat{I}_{T}$ and the optimal number of interior knots $\hat{K}$,
we then select tuning parameters $\lambda_{T}$ and $\mu_{T}$ for the proposed two-stage model identification procedure.
Following \cite{fan2001}, we take $a=3.7$ and find optimal tuning parameters $\lambda_{T}$ and $\mu_{T}$ by BIC in two
steps.

First, to select optimal $\lambda_{T}$, we define
\[BIC_{1}\left(\lambda_{T}\right)=\log{\left(RSS_{1}/T\right)}+d_{1}\frac{\log{T}}{T}
+\left(p-d_{1}\right)\frac{\log{\left(T/J_{K,2}\right)}}{T/J_{K,2}},\]
where $RSS_{1}=\sum_{t=1}^{T}\big[Y_{t,T}-\hat{\alpha}_{0}^{P}(t/T)-
 \sum_{k=1}^{p}\hat{\alpha}_{k}^{P}(t/T)\hat{\beta}_{k}(X_{t,T}^{\left(k\right)})\big]^{2}$,
 $J_{K,2}=K+m_{2}$ is the number of B-spline basis functions adopted in the second step estimation and
$d_{1}$ is the number of linear additive terms, i.e., $
\parallel\left(\hat{\alpha}_{k}^{P}\right)'\parallel_{L_{2}}$ is sufficiently small, say, no larger than $10^{-6}.$

Second, to select optimal $\mu_{T}$, we define
\[BIC_{2}\left(\mu_{T}\right)=\log{\left(RSS_{2}/T\right)}+d_{2}\frac{\log{T}}{T}+\left(p-d_{2}\right)\frac{\log{\left(T/J_{K,1}\right)}}{T/J_{K,1}},\]
where $RSS_{2}=\sum_{t=1}^{T}\big[Y_{t,T}-\hat{\alpha}_{0}^{P}(t/T)-
 \sum_{k=1}^{p}\hat{\alpha}_{k}^{P}(t/T)\hat{\beta}_{k}^{P}(X_{t,T}^{\left(k\right)})\big]^{2}$,
 and
$d_{2}$ is the number of constant varying-coefficient terms, i.e., $\parallel \beta_{k}''\parallel_{L_{2}}$ is sufficiently small.

Thus, we select the optimal tuning parameters
\[\hat{\lambda}_{T}=\operatorname*{argmin}_{\lambda_{T}}BIC_{1}\left(\lambda_{T}\right)\ \ \text{and}\ \
\hat{\mu}_{T}=\operatorname*{argmin}_{\mu_{T}}BIC_{2}\left(\mu_{T}\right).\]

\section{Asymptotic Results} \ \\

In this section, we demonstrate that the proposed three-step spline estimation method is consistent under regularity conditions and show that the proposed  two stage model identification procedure can correctly select additive terms and varying-coefficient terms with probability approaching one.  Furthermore, we conclude that $L_{2}$ rate of convergence of each component function achieves the optimal rate of the spline estimator of univariate nonparametric function stated in \cite{stone1982}. The regularity conditions and assumptions are given in the Appendix.

\subsection{Asymptotic results of spline estimators} \ \\

Let $G_{k}=span\{\psi_{kl},l=1,\cdots,J_{k,A}\}$ and $H_{k}=span\{\varphi_{kl},l=1,\cdots,J_{k,C}\}$.  We introduce
 \[\rho_{A}=\sum_{k=1}^{p}\inf_{\mu_{k}\in G_{k}}\sup_{x\in[0,1]}|\beta_{k}\left(x\right)-\mu_{k}\left(x\right)|,\ \
\rho_{C}=\sum_{k=1}^{p}\inf_{\nu_{k}\in H_{k}}\sup_{x\in[0,1]}|\alpha_{k}\left(x\right)-\nu_{k}\left(x\right)|\]
to measure the degree of spline approximation  of varying-coefficient component function and additive component function.

Proposition 1 establishes $L_{2}$ rate  of convergence of initial estimators of scaled additive component functions
$\gamma_{k}=w_{k}\beta_{k}.$
\begin{proposition}\label{prop1}
Under Assumptions (A1), (A2), (A4), (A5), (A7) , (A8)  and (A9), if $\rho_{A}=o\left(1\right),$ as $T\to\infty,$
\[\parallel\hat{\gamma}_{k}-\gamma_{k}\parallel^{2}_{L_{2}}=O_{p}\left(\frac{\rho_{A}^{2}}{N_{T}}+\frac{K_{A}}{T}\right),\]
\[\frac{1}{T}\sum_{t=1}^{T}\big[\hat{\gamma}_{k}(X_{t,T}^{\left(k\right)})-\gamma_{k}(X_{t,T}^{\left(k\right)})\big]^{2}
=O_p\left(\frac{\rho_{A}^{2}}{N_{T}}+\frac{K_{A}}{T}\right),\]
\end{proposition}

\vskip 2pt

{\bf Remark 3:}
In comparison with the convergence of the spline estimation of univariate nonparametric function,
we notice that the bias term in  $L_{2}$ rate  of convergence of initial estimators
is smaller when the number of groups $N_{T}$ is larger than 1.
The larger number of segmentation groups, the smaller the bias, given the number of observations in each group is at least larger than the number of parameters in spline approximation of $\gamma_{k}.$

\medskip

Based on the result of Proposition 1, one can construct $L_{2}$ rate of convergence of the spline estimation of varying-coefficient component function $\alpha_{k}$ as follows.
\begin{thm}\label{alpha}
Under Assumptions (A1) \textrm{-} (A10), if $\rho_{A}\vee \rho_{C}=o\left(1\right),$ as $T\to\infty,$
\[\parallel\hat{\alpha}_{k}-\alpha_{k}\parallel^{2}_{L_{2}}=O_{p}\left(\frac{\rho_{A}^{2}}{N_{T}}+\rho_{C}^{2}+\frac{K_{A}\vee K_{C}}{T}\right),\]
\[\frac{1}{T}\sum_{t=1}^{T}\big[\hat{\alpha}_{k}\left(t/T\right)-\alpha_{k}\left(t/T\right)\big]^{2}
=O_p\left(\frac{\rho_{A}^{2}}{N_{T}}+\rho_{C}^{2}+\frac{K_{A}\vee K_{C}}{T}\right),\]
where `$a \vee b $' denotes the maximum of $a$ and $b.$
\end{thm}

{\bf Remark 4:}
Theorem \ref{alpha} shows that there exists an additional bias term $\frac{\rho_{A}^{2}}{N_{T}}+\frac{K_{A}}{T}$ in comparison with the convergence of the spline estimation of univariate nonparametric function. This term happens to be the rate of convergence obtained in Proposition 1 and reflects the error of the initial estimator of scaled additive function $\gamma_{k}.$

\medskip

Next theorem presents $L_{2}$ rate of convergence of the spline estimation of additive component function $\beta_{k}.$
\begin{thm}\label{beta}
Under the Assumptions of Theorem 1,
\begin{equation*}
\begin{aligned}
\parallel\hat{\beta}_{k}-\beta_{k}\parallel^{2}_{L_{2}}
=&O_{p}\left(\rho_{A}^{2}+\rho_{C}^{2}+\frac{K_{A}\vee K_{C}}{T}
+\frac{K_{A}\rho_{A}^{2}}{TN_{T}}+\frac{K_{A}\rho_{C}^{2}}{T}\right)\\
=&O_{p}\left(\rho_{A}^{2}+\rho_{C}^{2}+\frac{K_{A}\vee K_{C}}{T}\right).
\end{aligned}
\end{equation*}
\end{thm}

{\bf Remark 5:}
Similarly, in comparison with the rate of convergence of spline estimation for univariate nonparametric function,
Theorem \ref{beta} also has an additional bias term $\rho_{C}^2+K_{C}/T$. The reason this term exists is because
the estimation of additive component function $\beta_{k}$ in Step III is based on the spline estimates of varying-coefficient function $\alpha_{k}$ obtained in Step II.  As expected, the convergence of the updated spline estimation of $\beta_{k}$ doe not depend on the number of segmentation groups in Step I of the initial estimation of rescaled additive function $\gamma_{k}.$

\subsection{Asymptotic results of model identification} 
We, here, respectively, demonstrate the consistency of selecting additive terms and varying-coefficient terms, and present the $L_{2}$ rate of convergence of penalized spline estimators of $\alpha_{k}$ and $\beta_{k}.$

\begin{thm}\label{Cadditive}
Suppose that $\alpha_{k}'=0,d_{k}\ge 2,k=p_1+p_2+1,\cdots,p.$  Given $\lambda_{T}\to 0,$
$\liminf_{T\to\infty}\liminf_{w\to 0+}p'_{\lambda_{T}}\left(w\right)>0$
and $\theta_{T}/\lambda_{T}\to 0$ with $\theta_{T}=K_{C}^{1/2}T^{-1/2}+K_{C}^{-d_{C}}$,
then, under Assumptions (A1) \textrm{-} (A10), as $T\to\infty,$
\begin{itemize}
\item[(i)]with probability approaching to 1, $\hat{\alpha}^{P}_{k}$ is constant a function a.s. for
$k=p_{1}+p_{2}+1,\cdots,p;$
\item[(ii)]$L_{2}$ rate of convergence for  penalized spline estimator of $\alpha_{k}$ is given by
\[
\parallel \hat{\alpha}_{k}^{P}-\alpha_{k}\parallel_{L_{2}}^{2}=O_p\left(K_{k,C}T^{-1}+K_{k,C}^{-2d_{k,C}}\right)
=O_p\left(K_{C}T^{-1}+\underline{K}_{C}^{-2d_{C}}\right)
\]
for $k=0,\cdots,p,$ where $\underline{K}_{C}=\min_{0\leq k\leq p}K_{k,C}$ and $d_{C}=\min_{0\leq k\leq p}d_{k}.$
\end{itemize}
\end{thm}
\begin{thm}\label{Cvarying}
Suppose that $\beta_{k}''=0,r_{k}\ge 2,k=p_1+1,\cdots,p_{1}+p_{2}.$  Given $\mu_{T}\to 0,$
$\liminf_{T\to\infty}\liminf_{w\to 0+}p'_{\mu_{T}}\left(w\right)>0$
and $\vartheta_{T}/\mu_{T}\to 0$ with $\vartheta_{T}=K_{A}^{1/2}T^{-1/2}+K_{A}^{-r_{A}},$
then, under Assumptions (A1) \textrm{-} (A10), as $T\to\infty,$
\begin{itemize}
\item [(i)]  with probability approaching to 1, $\hat{\beta}^{P}_{k}$ is a linear function a.s. for
$k=p_{1}+1,\cdots,p_{1}+p_{2};$
\item[(ii)]  $L_{2}$ rate of convergence for penalized spline estimator of $\beta_{k}$ is given by
\[\parallel \hat{\beta}_{k}^{P}-\beta_{k}\parallel_{L_{2}}^{2}=O_p\left(K_{k,A}T^{-1}+K_{k,A}^{-2r_{k,A}}\right)
=O_p\left(K_{A}T^{-1}+\underline{K}_{A}^{-2r_{A}}\right)\]
for $ k=1,\cdots,p,$ where $\underline{K}_{A}=\min_{1\leq k\leq p}K_{k,A}$ and $r_{A}=\min_{1\leq k\leq p}r_{k}.$
\end{itemize}
\end{thm}

{\bf Remark 6:}
Theorems \ref{Cadditive} and  \ref{Cvarying} show that  the penalized spline estimators of varying-coefficient component function $\alpha_{k}$
and additive component function $\beta_{k}$ both have the same   $L_{2}$  rate of convergence as that of the spline estimator of univariate nonparametric function.

\section{Numerical Studies} 
We consider two  simulation examples to illustrate the finite sample performance of the proposed
three-step spline estimation method and two-stage model selection procedure, respectively.

\subsection{Simulation Examples}

\begin{example}
The data are generated from the varying-coefficient additive model as follows
\begin{equation*}
\left\{
\begin{aligned}
Y_{t,T} = &  \alpha_{0}\left(t/T\right)+\alpha_{1}\left(t/T\right)\beta_{1}(X_{t,T}^{\left(1\right)})
+\alpha_{2}\left(t/T\right)\beta_{2}(X_{t,T}^{\left(2\right)})+\varepsilon_{t}\\
X_{t,T}^{\left(1\right)} = &0.6\left(t/T\right)X_{t-1,T}^{\left(1\right)}+0.5\zeta_{1,t},\\
X_{t,T}^{\left(2\right)} = & 0.9\left(t/T\right)X_{t-1,T}^{\left(2\right)}-0.6\left(t/T\right)^2X_{t-2,T}^{\left(2\right)} +0.5\zeta_{2,t},
\end{aligned}
\right.
 \end{equation*}
 where $\{\varepsilon_t\},$ $\{\zeta_{1,t}\}$ and $\{\zeta_{2,t}\}$ are iid standard normal variables and
 \begin{equation*}
\begin{aligned}
\alpha_{0}\left(u\right)=&1.5u+2\cos{\left(2\pi u\right) },\\
\alpha_{1}\left(u\right)=&\left(2u\sin{\left(2\pi u\right)}+1\right)/\parallel 2u\sin{\left(2\pi u\right)}+1\parallel_{L_{2}},\\
\alpha_{2}\left(u\right)=&\{3\left(1-u\right)^2\cos{\left(2\pi u\right)}+1\}/\parallel 3\left(1-u\right)^2\cos{\left(2\pi u\right)}+1\parallel_{L_{2}},\\
\beta_{1}\left(x_{1}\right)=&0.7\sin{\left(\frac{\pi }{2}x_{1}\right)}-0.5x_{1}\left(2-x_{1}\right)^{2},\\
\beta_{2}\left(x_2\right)=&2x_2cos\left(\frac{\pi }{2}x_2\right)-3.5\sin{\left(\frac{\pi}{2} x_{2}\right)}.
\end{aligned}
\end{equation*}
\end{example}

 To appraise the performance of the proposed three-step spline estimators,
 we use the mean integrated squared error (MISE) based on $Q=500$ Monte Carlo
 replications, that is
 \[MISE\left(\alpha_{k}\right)=\frac{1}{Q}\sum_{q=1}^{Q}\int\big[\hat{\alpha}_{k,q}\left(u\right)
 -\alpha_{k}\left(u\right)\big]^{2}\mathrm{d}u,\ \ k=0,1,2,\]
 \[MISE\left(\beta_{k}\right)=\frac{1}{Q}\sum_{q=1}^{Q}\int\big[\hat{\beta}_{k,q}\left(x_{k}\right)
 -\beta_{k}\left(x_{k}\right)\big]^{2}\mathrm{d}x_{k},\ \ k=1,2,\]
where $\hat{\alpha}_{k,q}\left(\cdot\right)$ and $\hat{\beta}_{k,q}\left(\cdot\right)$ are estimators of
$\alpha_{k}\left(\cdot\right)$ and $\beta_{k}\left(\cdot\right),$ respectively, in the $q$-th Monto Carlo sample.

The univariate nonparametric functions are approximated by B-spline of order of three or the quadratic splines. We consider the  sample size $T=300,600,900,$  the  number of interior knots $K=3,4$ and the segmentation length  $I=25,30$ in Step I  estimation.  We run the simulation for 500 times, and find out the MISE of three-step estimators decreases
as the sample size increases, regardless the values of $K$ and $I$.

Table \ref{tag:1} only gives the results of $K=4$ for different combinations of $I$ and $T$.
In addition, we list the MISE of oracle estimators, which refer to the spline estimator of $\alpha_{k}$
given all additive component functions are known in advance,
or correspondingly, the spline estimator of $\beta_{k}$ given all varying-coefficient component functions are known in advance.  As expected, MISE of oracle estimators for varying-coefficient components and additive components are better than
those of the proposed three-step spline estimators.
The last two columns in Table \ref{tag:1} depict the MISE of spline estimators for nonparametric component functions
in misspecified varying-coefficient model and misspecified additive model.
We note that they are discernibly larger than three-step spline estimators in varying-coefficient additive model.
\begin{table}[htbp]
\centering
\caption{Comparison of MISE of different estimators in Example 1 ($K=4$)}
\label{tag:1}
\vspace{3mm}
\begin{tabular}{cc|ccccc}
\hline
\hline
$T$ &$I$&$Function$ &Spline &Oracle&Varying-coefficient &Additive\\
\hline
\multirow{10}{*}{300} &\multirow{5}{*}{25}
&$\alpha_{0}$&0.1263&0.0379&0.5508&2.6448 \\
&&$\alpha_{1}$&0.0847&0.0197&3.8537&-\\
&&$\alpha_{2}$&0.0078&0.0070&3.9886&-\\
&&$\beta_{1}$&0.0896&0.0890&-&0.9831\\
&&$\beta_{2}$&0.0684&0.0644&-&0.5289\\
\cline{3-7}
&\multirow{5}{*}{30}
&$\alpha_{0}$&0.0961&0.0379&0.5508&2.6448 \\
&&$\alpha_{1}$&0.0925&0.0197&3.8537&-\\
&&$\alpha_{2}$&0.0077&0.0070&3.9886&-\\
&&$\beta_{1}$&0.0902&0.0890&-&0.9831\\
&&$\beta_{2}$&0.0695&0.0644&-&0.5289\\
\cline{1-7}
\multirow{10}{*}{600}   &\multirow{5}{*}{25}
&$\alpha_{0}$&0.0557&0.0224&0.2370&2.7392 \\
&&$\alpha_{1}$&0.0158&0.0138&3.8691&-\\
&&$\alpha_{2}$&0.0049&0.0045&3.9949&-\\
&&$\beta_{1}$&0.0486&0.0476&-&0.4024\\
&&$\beta_{2}$&0.0480&0.0452&-&0.4468\\
\cline{3-7}
&\multirow{5}{*}{30}
&$\alpha_{0}$&0.0478&0.0224&0.2370&2.7392 \\
&&$\alpha_{1}$&0.0139&0.0138&3.8691&-\\
&&$\alpha_{2}$&0.0049&0.0045&3.9949&-\\
&&$\beta_{1}$&0.0484&0.0476&-&0.4024\\
&&$\beta_{2}$&0.0493&0.0452&-&0.4468\\
\cline{1-7}
\multirow{10}{*}{900} &\multirow{5}{*}{25}
&$\alpha_{0}$ &0.0392&0.0172&0.2847&2.8345\\
&&$\alpha_{1}$&0.0107&0.0101&3.9651&-\\
&&$\alpha_{2}$&0.0041&0.0037&3.9945&-\\
&&$\beta_{1}$&0.0450&0.0442&-&0.4210\\
&&$\beta_{2}$&0.0377&0.0351&-&0.3983\\
\cline{3-7}
& \multirow{5}{*}{30}
&$\alpha_{0}$&0.0353&0.0172&0.2847&2.8345 \\
&&$\alpha_{1}$&0.0105&0.0101&3.9651&-\\
&&$\alpha_{2}$&0.0041&0.0037&3.9945&-\\
&&$\beta_{1}$&0.0450&0.0442&-&0.4210\\
&& $\beta_{2}$&0.0374&0.0351&-&0.3983\\
\hline\hline
\end{tabular}
\end{table}

To visualize the performance of three-step estimation method, we consider $T=500,K=3 \ \text{and}\  I=25,$
and approximate the unknown functions using three order B-spline functions.
Figures \ref{fig:1} and  \ref{fig:2} presents the estimated additive component function $\beta_{k}\left(\cdot\right),k=1,2,$
and  the estimated  varying-coefficient functions $\alpha_{k}\left(\cdot\right),$ $k=0,1,2$, respectively.
They both show that the  proposed three-step estimation method can approximate  the true function well
even for a moderate sample size.
\begin{figure}[h!]
\centering
\begin{tabular}{@{}c@{}c}
\includegraphics[width=0.53\linewidth]{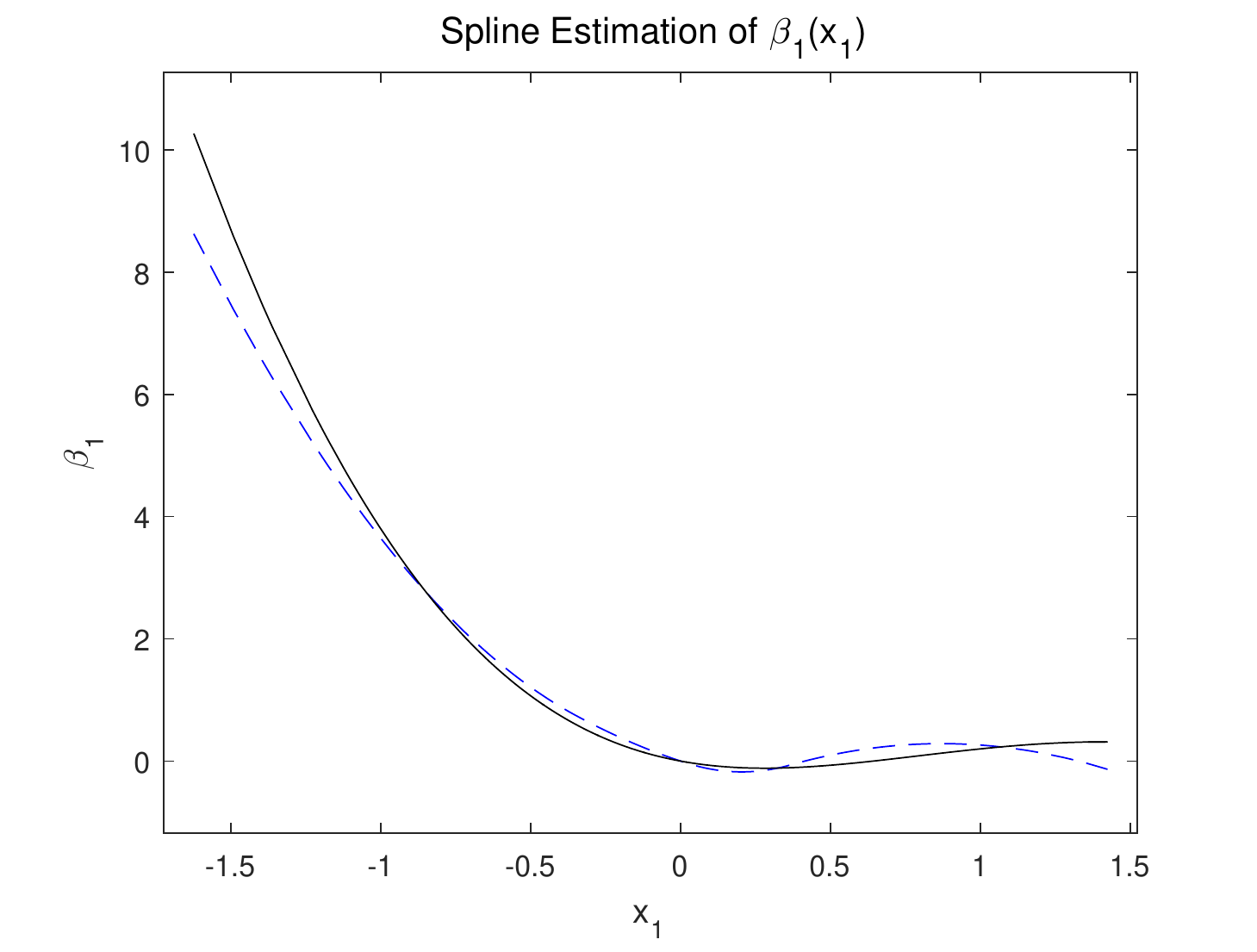} \ &
\includegraphics[width=0.53\linewidth]{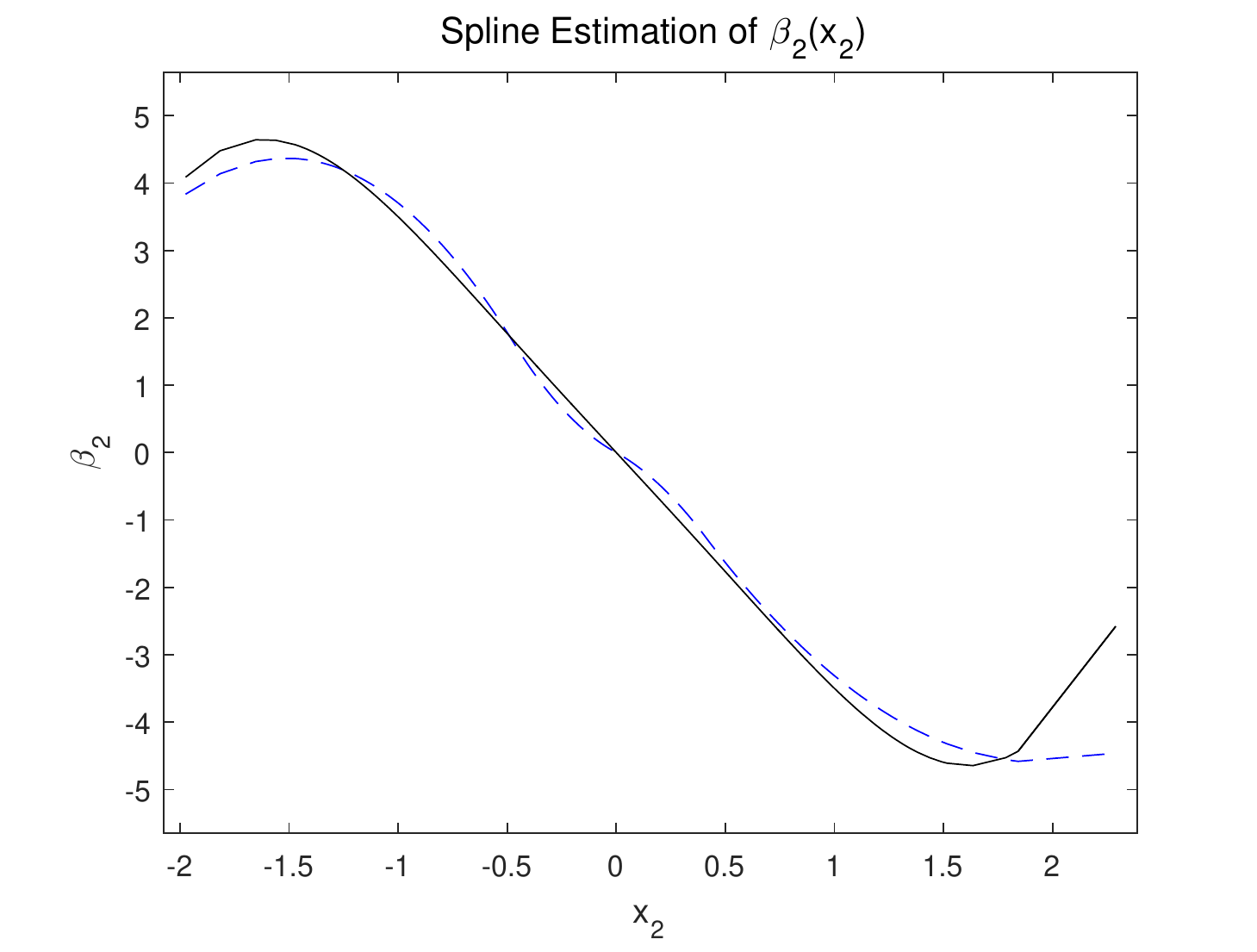} \ \\
{\small(a)} & {\small(b)} \
\end{tabular}
\caption{Estimation of  additive component functions.
(a) \textrm{-} (b) Additive component functions $\beta_{k}\left(\cdot\right),k=1,2$ (solid black curve)
and its estimation $\hat{\beta}_{k}\left(\cdot\right)$ (dashed blue curve).
\label{fig:1}}
\end{figure}

\begin{figure}[hbpt]
\centering
\begin{tabular}{@{}c@{}c@{}c}
\includegraphics[width=0.35\linewidth]{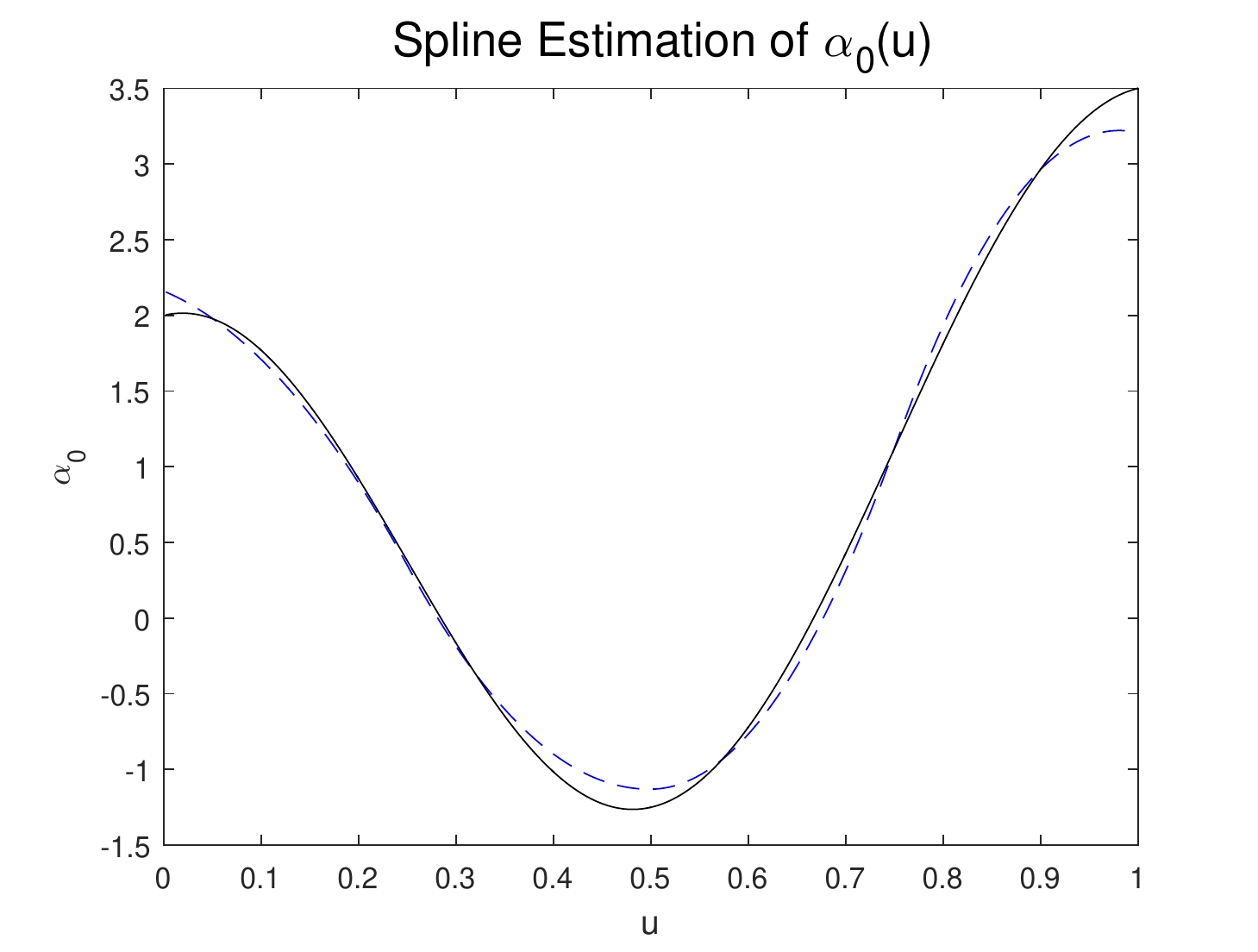} \ &
\includegraphics[width=0.35\linewidth]{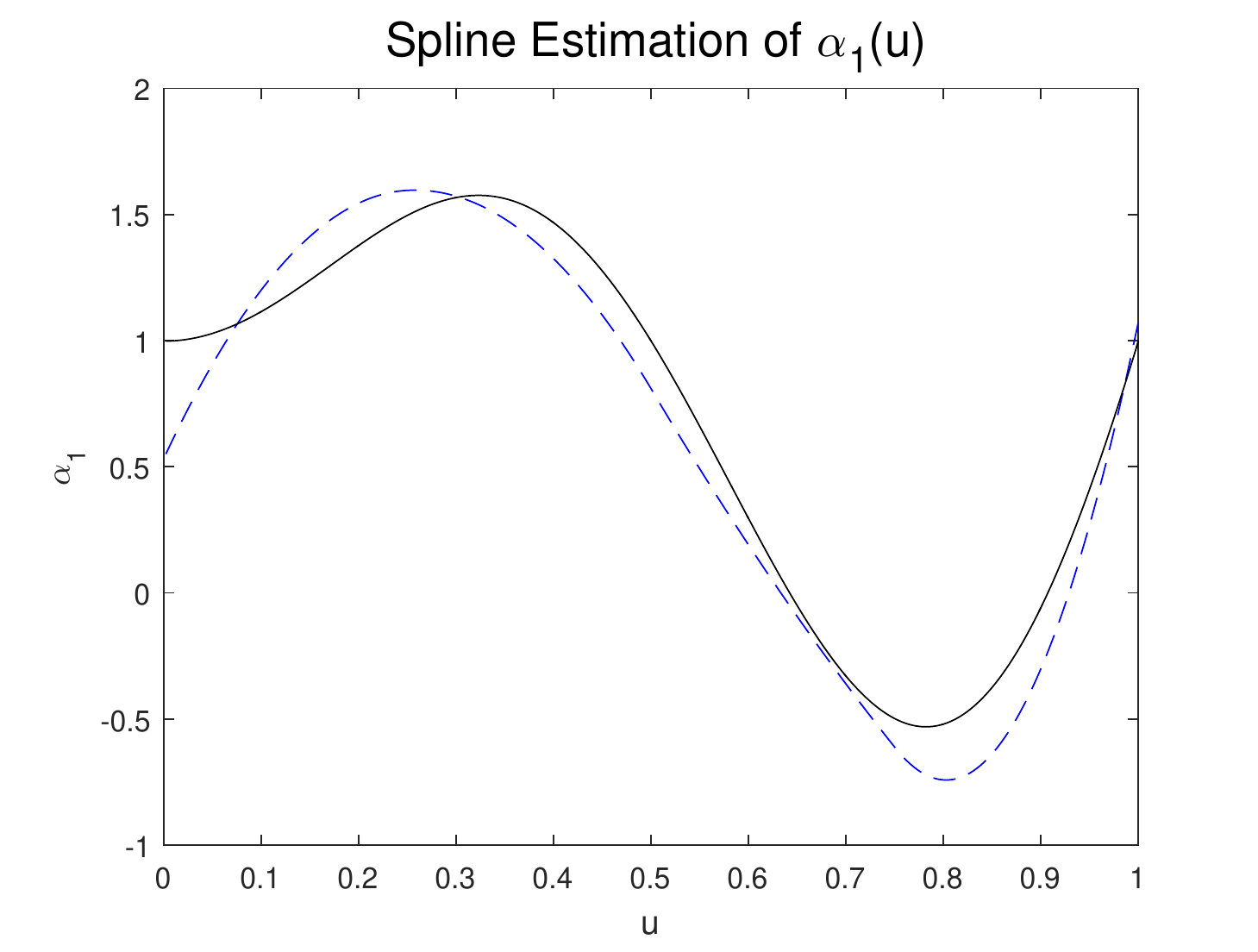} \ &
\includegraphics[width=0.35\linewidth]{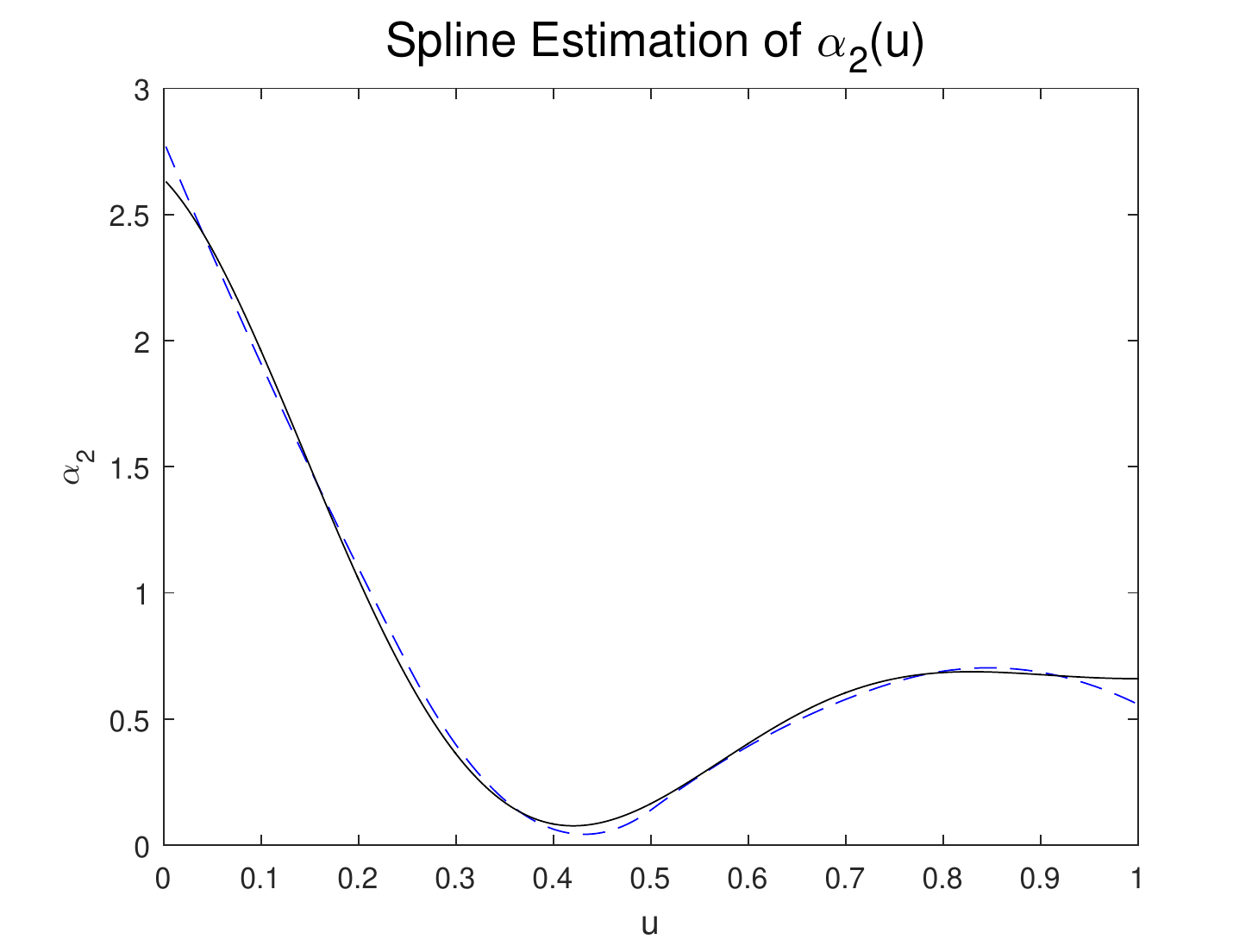} \ \\
{\small(c)} & {\small(d)} & {\small(e)}\
\end{tabular}
 \caption{Estimation of varying-coefficient  component functions.
 (c) \textrm{-} (e)Varying-coefficient  component functions $\alpha_{k}\left(\cdot\right),k=0,1,2$ (solid black curve) and their
 spline estimators $\hat{\alpha}_{k}$ (dashed blue curve).
 \label{fig:2}}
\end{figure}

\begin{example}
We consider a varying-coefficient additive model with additive term and varying-coefficient term.
\begin{equation*}
\begin{aligned}
Y_{t,T} = &  \alpha_{0}\left(t/T\right)+\sum_{k=1}^{4}\alpha_{k}\left(t/T\right)\beta_{k}(X_{t,T}^{\left(k\right)})+\varepsilon_{t}\\
X_{t,T}^{\left(1\right)} = &0.7\left(t/T\right)X_{t-1,T}^{\left(1\right)}-0.5\left(t/T\right)^2X_{t-2,T}^{\left(1\right)} +0.5\xi_{1,t},\\
X_{t,T}^{\left(2\right)} = & 0.8\left(t/T\right)X_{t-1,T}^{\left(2\right)}-0.2\left(t/T\right)^2X_{t-2,T}^{\left(2\right)} +0.5\xi_{2,t},\\
X_{t,T}^{\left(3\right)} = & 0.6\left(t/T\right)X_{t-1,T}^{\left(3\right)}-0.3\left(t/T\right)^2X_{t-2,T}^{\left(3\right)} +0.5\xi_{3,t},\\
X_{t,T}^{\left(4\right)} =& 0.6\left(t/T\right)X_{t-1,T}^{\left(4\right)}+0.5\xi_{4,t},
\end{aligned}
 \end{equation*}
 where $\{\varepsilon_t\},$ $\{\xi_{k,t}\},k=1,2,3,4,$ are iid standard normal variables and
 $a_{k}\left(u\right),k=0,1,2,$ are given in Example 1,  $\alpha_{3}\left(u\right)=1,$
%
\begin{equation*}
\begin{aligned}
\alpha_{4}\left(u\right)=&\{3u\left(1-u\right)^2+1\}/\parallel 3u\left(1-u\right)^2+1\parallel_{L_{2}},\\
\beta_{1}\left(x_{1}\right)=&0.7sin\left(\frac{\pi }{2}x_{1}\right)-0.5x_{1}\left(2-x_{1}\right)^2,\\
\beta_{2}\left(x_{2}\right)=&3x_{2}\cos{\left(\frac{\pi }{2}x_{2}\right)}-0.8\sin{\left(\frac{\pi }{2}x_{2}\right)},\\
\end{aligned}
\end{equation*}
$\beta_{3}\left(x_{3}\right)=2x_{3}\left(1+x_{3}\right)$ and $\beta_{4}\left(x_{4}\right)=x_{4}.$
\end{example}
To investigate the performance of the proposed two-stage model selection procedure, we take the sample size $T=300,600,900,$ segment length $I=30$, and the same interior knots $K=3$ for all univariate functions.
Based on 100 replications, Table \ref{tag:2} lists the MISE of three-step spline estimators (UMISE), the SCAD penalized spline estimators (PMISE)
and oracle estimators (OMISE) which are obtained by assuming the true model structure is known.
The smoothing parameter  $\left(\lambda_{T},\mu_{T}\right)$ is chosen according to the BIC criterion described in Section 4
and the number in parenthesis is the standard deviation of corresponding MISE.

Table \ref{tag:3} compares the performance of correct-fitting(C-F), over-fitting(O-F) and under-fitting(U-F)
the pure varying-coefficient terms (i.e., the additive component function is linear), the pure additive terms (i.e., the varying-coefficient component function is constant)
and the true model.
The results show  the numbers of correctly identifying additive terms, varying-coefficient terms and true model become larger
as the sample size increases.
\begin{table}[h!] 
\centering
\caption{Comparison of MISE  in Example 2 }
\label{tag:2}
\vspace{3mm}
\begin{tabular}{cccccc}
\hline
\hline
$T$ &$\left(\lambda_{T},\mu_{T}\right)$ &$Function$ &UMISE &PMISE&OMISE \\
\hline
\multirow{9 }{*}{300} &\multirow{9 }{*}{$\left(0.0686,0.15\right)$}
&$\alpha_{0}$&0.2284 \scriptsize{(0.1954)}&0.0900\scriptsize{(0.0615)}&0.0351\scriptsize{(0.0206)} \\
&&$\alpha_{1}$&0.0378\scriptsize{(0.0231)}&0.0260\scriptsize{(0.0166)}&0.0239\scriptsize{(0.0151)}\\
&&$\alpha_{2}$&0.1419\scriptsize{(0.0967)}&0.0571\scriptsize{(0.0406)}&0.0387\scriptsize{(0.0269)}\\
&&$\alpha_{3}$&0.0104\scriptsize{(0.0074)}&0.0044\scriptsize{(0.0081)}&--\\
&&$\alpha_{4}$&0.1142\scriptsize{(0.1160)}&0.0729\scriptsize{(0.0589)}&0.0550\scriptsize{(0.0358)}\\
&&$\beta_{1}$&0.0340\scriptsize{(0.0235)}&0.0310\scriptsize{(0.0224)}&0.0269\scriptsize{(0.0189)}\\
&&$\beta_{2}$&0.0698\scriptsize{(0.0455)}&0.0492\scriptsize{(0.0369)}&0.0313\scriptsize{(0.0288)}\\
&&$\beta_{3}$&0.0288\scriptsize{(0.0226)}&0.0288\scriptsize{(0.0244)}&0.0265\scriptsize{(0.0222)}\\
&&$\beta_{4}$&0.0303\scriptsize{(0.0241)}&0.0167\scriptsize{(0.0258)}&--\\
\hline
\multirow{9 }{*}{600} &\multirow{9 }{*}{$\left(0.053,0.1286\right)$}
&$\alpha_{0}$&0.1635 \scriptsize{(0.1666)}&0.0307\scriptsize{(0.0201)}&0.0160\scriptsize{(0.0074)} \\
&&$\alpha_{1}$&0.0114\scriptsize{(0.0064)}&0.0087\scriptsize{(0.0054)}&0.0083\scriptsize{(0.0057)}\\
&&$\alpha_{2}$&0.0769\scriptsize{(0.0679)}&0.0295\scriptsize{(0.0199)}&0.0171\scriptsize{(0.0111)}\\
&&$\alpha_{3}$&0.0074\scriptsize{(0.0045)}&0.0019\scriptsize{(0.0039)}&--\\
&&$\alpha_{4}$&0.0381\scriptsize{(0.0287)}&0.0312\scriptsize{(0.0214)}&0.0285\scriptsize{(0.0207)}\\
&&$\beta_{1}$&0.0195\scriptsize{(0.0142)}&0.0181\scriptsize{(0.0127)}&0.0175\scriptsize{(0.0126)}\\
&&$\beta_{2}$&0.0536\scriptsize{(0.0365)}&0.0321\scriptsize{(0.0276)}&0.0183\scriptsize{(0.0097)}\\
&&$\beta_{3}$&0.0134\scriptsize{(0.0092)}&0.0124\scriptsize{(0.0092)}&0.0121\scriptsize{(0.0090)}\\
&&$\beta_{4}$&0.0119\scriptsize{(0.0096)}&0.0049\scriptsize{(0.0104)}&--\\
\hline
\multirow{9 }{*}{900} &\multirow{9 }{*}{$\left(0.06,0.15\right)$}
&$\alpha_{0}$&0.0993 \scriptsize{(0.1047)}&0.0220\scriptsize{(0.0177)}&0.0114\scriptsize{(0.0054)} \\
&&$\alpha_{1}$&0.0111\scriptsize{(0.0053)}&0.0066\scriptsize{(0.0034)}&0.0063\scriptsize{(0.0032)}\\
&&$\alpha_{2}$&0.0409\scriptsize{(0.0235)}&0.0144\scriptsize{(0.0108)}&0.0113\scriptsize{(0.0073)}\\
&&$\alpha_{3}$&0.0046\scriptsize{(0.0031)}&0.0002\scriptsize{(0.0012)}&--\\
&&$\alpha_{4}$&0.0215\scriptsize{(0.0139)}&0.0205\scriptsize{(0.0110)}&0.0182\scriptsize{(0.0113)}\\
&&$\beta_{1}$&0.0150\scriptsize{(0.0089)}&0.0108\scriptsize{(0.0068)}&0.0098\scriptsize{(0.0062)}\\
&&$\beta_{2}$&0.0215\scriptsize{(0.0118)}&0.0205\scriptsize{(0.0125)}&0.0178\scriptsize{(0.0094)}\\
&&$\beta_{3}$&0.0081\scriptsize{(0.0054)}&0.0078\scriptsize{(0.0052)}&0.0075\scriptsize{(0.0047)}\\
&&$\beta_{4}$&0.0079\scriptsize{(0.0054)}&0.0015\scriptsize{(0.0025)}&--\\
\hline\hline
\end{tabular}
\end{table}
\begin{table}[h!]
\centering
\caption{Performance of Model identification in Example 2}
\label{tag:3}
\vspace{3mm}
\begin{tabular}{cccccccccc}
\hline
\hline
\multirow{2}{*}{$T$}
 & \multicolumn{3}{c}{Additive terms}
& \multicolumn{3}{c}{Varying-coefficient terms}
& \multicolumn{3}{c}{True model}\\
\cmidrule(lr){2-4}\cmidrule(lr){5-7}\cmidrule(lr){8-10}
&C-F & O-F &U-F
&C-F & O-F &U-F
&C-F & O-F &U-F\\
\hline
300&71&12&17&75&0&25&54&38&8\\
600&80&7&13&88&0&12&70&24&6\\
900&92&6&2&94&0&6&86&8&6\\
\hline\hline
\end{tabular}
\end{table}

\appendix

\section{Appendix }
\subsection{Assumption Sets}
Let $|\mathbf{a}|$ be the Euclidean norm of a real valued vector $\mathbf{a}.$
Denote the space of $l$-order smooth functions defined on $[0,1]$ as
$C^{\left(l\right)}[0,1]=\{m|m^{\left(l\right)}\in C[0,1]\}$ and the class of Lipschitz continuous functions for some fixed constant $C>0$ as
$Lip\left([0,1],C\right)=\{m||m\left(x\right)-m\left(x'\right)|\leq C|x-x'|,\forall x,x'\in[0,1]\}.$\\

The necessary conditions to prove asymptotic properties are listed as below.

\begin{itemize}
\item[ (A1)] The process $\{\mathbf{X}_{t,T}\}$ is stationary locally in time, that is, for each rescaled time point $u\in[0,1],$
there exists a strictly stationary process $\{\mathbf{X}_{t}\left(u\right)\}$ such that
$| \mathbf{X}_{t,T}-\mathbf{X}_{t}\left(u\right)|\leq\{|\frac{t}{T}-u|+\frac{1}{T}\}U_{t,T}\left(u\right)$ a.s.
with $E[U_{t,T}\left(u\right)^{\rho}]\leq C$ for some $\rho\ge2$ and $C<\infty$ independent of $u,t,$ and $T.$
\item[(A2)] At each rescaled time point $u\in[0,1],$ the joint density function $f_{u}\left(\mathbf{x}\right)$ of the stationary approximation process
$\mathbf{X}_{t}\left(u\right)$ is  bounded below and above  uniformly on $u\in[0,1]:$
\[0<c_{f}\leq\inf_{\mathbf{x}\in[0,1]^{p}}f_{u}\left(\mathbf{x}\right)\leq\sup_{\mathbf{x}\in[0,1]^{p}}f_{u}\left(\mathbf{x}\right)\leq C_{f}
\ \text{uniformly on}\ u\in[0,1].\] Meantime, $\mathbf{X}_{t,T}$ has density function with respective to certain measure.
\item[(A3)] $\varepsilon_{t}$ iid, and $E\varepsilon_{t}=0,$ $E\varepsilon_{t}^2=1.$
Given $\mathbf{X}_{t,T},$ $\varepsilon_{t}\sim WN\left(0,1\right).$
\item[(A4)] There exists positive constants $K_{0}$ and $\lambda$ such that $\alpha\left(k\right)\leq K_{0} \exp^{-\lambda k}$  for all $k\ge 1,$
where $\alpha\left(\cdot\right)$ is  the $\alpha$-mixing coefficients for process
$\big\{\mathcal{Z}_t:=\left(\mathbf{X}_{t,T}^\tau, \varepsilon _t\right)^\tau,t=1,\cdots,T\big\},$
and defined as
\[\alpha\left(k\right)=\sup_{B\in\sigma\{\mathcal{Z}_s,s\leq t\}, C\in \sigma\{\mathcal{Z}_s,s\ge t+k\}}
|P\left(B\cap C\right)-P\left(B\right)P\left(C\right)|,\ k\ge 1.\]
\item[(A5)] The conditional standard deviation function $\sigma$ is bounded below and above uniformly on $u\in[0,1],$ i.e.,
\[c_{\sigma}\leq\inf_{\mathbf{x}\in[0,1]^{p}}\sigma\left(u,\mathbf{x}\right)
\leq \sup_{\mathbf{x}\in[0,1]^{p}}\sigma\left(u,\mathbf{x}\right)\leq C_{\sigma}\ \text{uniformly on}\  u\in[0,1] \]
for some positive constants $c_{\sigma}$ and  $C_{\sigma}.$
\item[(A6)] $\alpha_{k}\in C^{\left(d_{k}-1\right)}[0,1]$ and $\alpha_{k}^{\left(d_{k}-1\right)}\in Lip([0,1],C_{1,k})$
where $d_{k}$ is  an integer such that $1\leq d_{k}\leq q_{k}.$
\item[(A7)]  $\beta_{k}\in C^{\left(r_{k}-1\right)}[0,1]$ and $\beta_{k}^{\left(r_{k}-1\right)}\in Lip([0,1],C_{2,k})$
where $r_{k}$ is  an integer such that $ 1\leq r_{k}\leq p_{k}.$
\item[(A8)] The knots for $\alpha_{k}\left(\cdot\right),k=0,\cdots,p,$
\begin{equation*}
\begin{aligned}
0=\eta_{k,1-q_{k}}=&\cdots=\eta_{k,0}<\eta_{k,1}<\cdots<\eta_{k,K_{k,C}}<\eta_{k,K_{k,C}+1}\\=&\cdots=\eta_{k,K_{k,C}+q_{k}}=1
\end{aligned}
\end{equation*}
and the knots for $\beta_{k}\left(\cdot\right),k=1,\cdots,p ,$
\begin{equation*}
\begin{aligned}
0=\tau_{k,1-p_{k}}=&\cdots=\tau_{k,0}<\tau_{k,1}<\cdots<\tau_{k,K_{k,A}}<\tau_{k,K_{k,A}+1}\\=&\cdots=\tau_{k,K_{k,A}+p_{k}}=1
\end{aligned}
\end{equation*}
has bounded mesh ratio:
\[\limsup_{T\to\infty}\max_{0\leq k\leq p}\frac{\max_{1\leq l\leq K_{k,C}+1}\{\eta_{k,l}-\eta_{k,l-1}\}}{\min_{1\leq l\leq K_{k,C}+1}\{\eta_{k,l}-\eta_{k,l-1}\}}<\infty.\]
\[\limsup_{T\to\infty}\max_{1\leq k\leq p}\frac{\max_{1\leq l\leq K_{k,A}+1}\{\tau_{k,l}-\tau_{k,l-1}\}}{\min_{1\leq l\leq K_{k,A}+1}\{\tau_{k,l}-\tau_{k,l-1}\}}<\infty,\]

\item[(A9)] $\limsup_{T\to\infty}\left(K_{A}/\min_{1\leq k\leq p} K_{k,A}\right)<\infty$, $K_{A}/T\to 0,$
$K_{A}\log{T}/\sqrt{T}\to 0$ and $\sqrt{T}K_{A}\to\infty.$
\item[(A10)] $\limsup_{T\to\infty}\left(K_{C}/\min_{0\leq k\leq p} K_{k,C}\right)<\infty$ and $K_{C}/T\to 0.$
\end{itemize}

{\bf Remark 7:}
Assumption (A1) specifies a data generating process (GDP).
Assumptions (A2) - (A5) are standard in time-series context, see \cite{vogt2012nonparametric,wang2007spline}.
Assumptions (A6) - (A8) are common in typical spline approximation literature, for instance, \cite{zhang2015functional}.

\subsection{Proofs for  Main Theorems} 
\medskip

Firstly, we need some lemmas before proving main theorems.
Let $\mathbf{g}=\left(g_{1},\cdots,g_{p}\right)^{\tau}$ and $\mathbf{h}=\left(h_{1},\cdots,h_{p}\right)^{\tau}$
be  any two-vector valued function. Define empirical inner product
\[\langle\mathbf{g},\mathbf{h}\rangle_{T}=\frac{1}{T}\sum_{t=1}^{T}\sum_{k=1}^{p}g_{k}\left(X_{t,T}^{\left(k\right)}\right)
h_{k}\left(X_{t,T}^{\left(k\right)}\right),\]
and theoretical inner product
\[\langle\mathbf{g},\mathbf{h}\rangle=\sum_{k=1}^{p}\int_{0}^{1}E\big[g_{k}\left(X_{t}^{\left(k\right)}\left(u\right)\right)
h_{k}\left(X_{t}^{\left(k\right)}\left(u\right)\right)\big]\mathrm{d}u,\]
Denote the induced norm by $\langle\mathbf{g},\mathbf{h}\rangle_{T}$ and $\langle\mathbf{g},\mathbf{h}\rangle$
as $\parallel\cdot\parallel_{T}$ and $\parallel\cdot\parallel,$ respectively.
In addition to, $\parallel\mathbf{g}\parallel_{L_{2}}^{2}=\sum_{k=1}^{p}\parallel g_{k}\parallel_{L_{2}}^{2}.$
Given sequences of positive numbers $a_{n}$ and $b_{n},$ $a_{n}\preceq b_{n}$ means $a_{n}/b_{n}$ is bounded and
$a_{n}\asymp b_{n}$ means $a_{n}\preceq b_{n}$ and $b_{n}\preceq a_{n}$ hold.
\begin{lemma}\label{inducenorm}
Let $g_{k}\left(z_{k}\right)=\sum_{l=1}^{J_{k,A}}\gamma_{kl}\psi_{kl}\left(z_{k}\right)$ and
$\gamma_{k}=\left(\gamma_{k1},\cdots,\gamma_{kJ_{k,A}}\right)^{\tau}$ for $k=1,\cdots,p.$ Denote
$\mathbf{\gamma}=\left(\gamma_{1}^{\tau},\cdots,\gamma_{p}^{\tau}\right)^{\tau}$ and
$\mathbf{g}\left(\mathbf{z}\right)=\left(g_{1}\left(z_{1}\right),\cdots,g_{p}\left(z_{p}\right)\right)^{\tau}.$
Then $\parallel \mathbf{g}\parallel^2\asymp\parallel \mathbf{g}\parallel_{L_{2}}^2 \asymp|\mathbf{\gamma}|^2$
holds under Assumption (A2).
\end{lemma}
\begin{proof} In combination with Assumption (A2) and the property (iv) of $\{\psi_{kl}\},$ we have
\begin{equation*}
\begin{aligned}
\parallel \mathbf{g}\parallel^{2}=&\sum_{k=1}^{p}\int_{0}^{1}E\big[g_{k}^{2}(X_{t}^{\left(k\right)}\left(u\right))\big]\mathrm{d}u
=\sum_{k=1}^{p}\int_{0}^{1}E\big[\sum_{l=1}^{J_{k,A}}\gamma_{kl}\psi_{kl}(X_{t}^{\left(k\right)}\left(u\right))\big]^{2}\mathrm{d}u\\
\asymp&\sum_{k=1}^{p}\int_{0}^{1}\parallel g_{k}\parallel^{2}_{L_{2}}\mathrm{d}u=\parallel\mathbf{g}\parallel_{L_{2}}^{2}
\asymp\sum_{k=1}^{p}|\gamma_{k}|^{2}=|\mathbf{\gamma}|^{2}
\end{aligned}
\end{equation*}
\end{proof}

\begin{lemma}\label{inner}
Let $\mathcal{G}$ be the collection of vector valued functions $\mathbf{g}=\left(g_{1},\cdots,g_{p}\right)^{\tau}$
with $g_{k}\in G_{k}$ such that $\parallel g_{k}\parallel_{L_{2}}<\infty$  for $k=1,\cdots,p,$
where $G_{k}$ is defined in Section 5.1. Then under Assumption (A1), (A2), (A4) and (A9), as $T\to\infty,$
\[\sup_{\mathbf{g}_{1},\mathbf{g}_{2}\in\mathcal{G}}\frac{|\langle \mathbf{g}_{1},\mathbf{g}_{2}\rangle_{T}
-\langle \mathbf{g}_{2},\mathbf{g}_{2}\rangle|}{\parallel  \mathbf{g}_{1}\parallel\cdot \parallel  \mathbf{g}_{2}\parallel}
=O_p\left(\frac{K_{A}\log{T}}{T}\right)=o_p\left(1\right).\]
\end{lemma}
\begin{proof} For any $\mathbf{g}^{\left(i\right)}=(g_{1}^{\left(i\right)},\cdots,g_{p}^{\left(i\right)})^{\tau}\in\mathcal{G},i=1,2,$
there exists coefficients $\gamma_{kl}^{\left(1\right)}$
and $\gamma_{kl}^{\left(2\right)},$ $l=1,\cdots,J_{k,A},k=1,\cdots,p,$ such that
\[\mathbf{g}^{\left(1\right)}_{k}\left(z_{k}\right)=\sum_{l=1}^{J_{k,A}}\gamma_{kl}^{\left(1\right)}\psi_{kl}\left(z_{k}\right)\ \ \text{and}\ \
\mathbf{g}^{\left(2\right)}_{k}\left(z_{k}\right)=\sum_{l=1}^{J_{k,A}}\gamma_{kl}^{\left(2\right)}\psi_{kl}\left(z_{k}\right)\]
for $k=1,\cdots,p.$
 It is not difficult to see that
\[
|\langle \mathbf{g}_{1},\mathbf{g}_{2}\rangle_{T}-\langle \mathbf{g}_{1},\mathbf{g}_{2}\rangle|\leq\sum_{k=1}^{p}\sum_{l,l'=1}^{J_{k,A}}|\gamma_{kl}^{\left(1\right)}\gamma_{kl'}^{\left(2\right)}|
\cdot|\langle\psi_{kl},\psi_{kl'}\rangle_{T}-\langle\psi_{kl},\psi_{kl'}\rangle|.
\]
For any given $k=1,\cdots,p,$ let
$l'\in A_{k}\left(l\right)$ if the intersection of the supports of $\psi_{kl}$ and $\psi_{kl'}$ contains an open interval.
That is, $\langle \psi_{kl},\psi_{kl'}\rangle_{T}=\langle \psi_{kl},\psi_{kl'}\rangle=0$ if $l'\notin A_{k}\left(l\right).$
Moreover, it is known $\#\{A_{k}\left(l\right)\}\leq C$ for some constant $C$ and all $l,k.$ Moveover, for any given
$k,l$ and $l'\in A_{k}\left(l\right),$ we have
\begin{equation*}
\begin{aligned}
&|\langle \psi_{kl},\psi_{kl'}\rangle_{T}-\langle \psi_{kl},\psi_{kl'}\rangle|\\
=&|T^{-1}\sum_{t=1}^{T}\psi_{kl}(X_{t,T}^{\left(k\right)})\psi_{kl'}(X_{t,T}^{\left(k\right)})
-\int_{0}^{1}E[\psi_{kl}(X_{t}^{\left(k\right)}\left(u\right))\psi_{kl'}(X_{t}^{\left(k\right)}\left(u\right))]\mathrm{d}u|\\
\leq &T^{-1}\sum_{t=1}^{T}|\psi_{kl}(X_{t,T}^{\left(k\right)})\psi_{kl'}(X_{t,T}^{\left(k\right)})
-\psi_{kl}(X_{t}^{\left(k\right)}(t/T))\psi_{kl'}(X_{t}^{\left(k\right)}(t/T))|\\
+&T^{-1}\sum_{t=1}^{T}|\psi_{kl}(X_{t}^{\left(k\right)}(t/T))\psi_{kl'}(X_{t}^{\left(k\right)}(t/T))
-E[\psi_{kl}(X_{t}^{\left(k\right)}(t/T))\psi_{kl'}(X_{t}^{\left(k\right)}(t/T))]|\\
+&|T^{-1}\sum_{t=1}^{T}E[\psi_{kl}(X_{t}^{\left(k\right)}\left(t/T\right))\psi_{kl'}(X_{t}^{\left(k\right)}\left(t/T\right))]
-\int_{0}^{1}E[\psi_{kl}(X_{t}^{\left(k\right)}\left(u\right))\psi_{kl'}(X_{t}^{\left(k\right)}\left(u\right))]\mathrm{d}u|.
\end{aligned}
\end{equation*}
By Assumption (A1) and the boundness of B-spline, the first term above is bounded by $O_p\left(K_{A}/T\right).$
Employing Berstein's inequality, the second term is bounded by $O_p(K_{A}\log{T}/\sqrt{T}).$
The last term is bounded by $O\left(1/T\right)$  from the integral theory.
Therefore, using Cauchy-Schwartz inequality and Assumption (A9), we obtain that
\begin{equation*}
\begin{aligned}
|\langle \mathbf{g}_{1},\mathbf{g}_{2}\rangle_{T}-\langle \mathbf{g}_{1},\mathbf{g}_{2}\rangle|=&
O_p\left(K_{A}\log{T}/\sqrt{T}\right)\sum_{k=1}^{p}\sum_{l,l'}|\gamma_{kl}^{\left(1\right)}|\cdot
|\gamma_{kl'}^{\left(2\right)}|I\left(l'\in A_{k}\left(l\right)\right)\\
=&O_p\left(K_{A}\log{T}/\sqrt{T}\right)|\mathbf{\gamma}^{\left(1\right)}|\cdot|\mathbf{\gamma}^{\left(2\right)}|,
\end{aligned}
\end{equation*}
where $\mathbf{\gamma}^{\left(1\right)}$ and $\mathbf{\gamma}^{\left(2\right)}$ denote the vectors with entries
$\gamma_{kl}^{\left(1\right)}$ and $\gamma_{kl}^{\left(2\right)},$ respectively.
By Lemma \ref{inducenorm}, we see $|\mathbf{\gamma}^{\left(i\right)}|\asymp\parallel\mathbf{g}_{i}\parallel,i=1,2$
which completes the proof.
\end{proof}
\begin{lemma} \label{eigenvalue}
Under Assumption (A1), (A2)  (A4) and (A9),  as $T\to\infty,$
\begin{itemize}
\item[(i)] For each $s=1,\cdots,N_{T},$ $I_{T}^{-1}\sum_{t=1}^{I_{T}}V_{st}V_{st}^\tau$ has eigenvalues bounded away from 0 and $\infty$
with probability tending to one;
\item[(ii)] $T^{-1}\sum_{t=1}^{T}\Psi_{k}(X_{t,T}^{\left(k\right)})\Psi_{k}^{\tau}(X_{t,T}^{\left(k\right)})$
has eigenvalues bounded away from 0 and $\infty,$ with probability tending to one.
\end{itemize}
\end{lemma}
\begin{proof} We only show (ii), the proof of (i) is similar.
For any given vector $\mathbf{\gamma}=\left(\gamma_{1}^{\tau},\cdots,\gamma_{p}^{\tau}\right)^{\tau}$
with $\gamma_{k}=\left(\gamma_{k1},\cdots,\gamma_{kJ_{k,A}}\right)^{\tau},$ $k=1,\cdots ,p,$
 let $g_{k}\left(z_{k}\right)=\sum_{l=1}^{J_{k,A}}\gamma_{kl}\psi_{kl}\left(z_{k}\right)$
and $\mathbf{g}\left(\mathbf{z}\right)=\left(g_{1}\left(z_{1}\right),\cdots,g_{p}\left(z_{p}\right)\right)^{\tau},$
where $\mathbf{z}=\left(z_{1},\cdots,z_{p}\right).$
 Denote $\tilde{\Psi}\left(\mathbf{z}\right)=\left(\Psi_{1}^{\tau}\left(z_{1}\right),\cdots,\Psi_{p}^{\tau}\left(z_{p}\right)\right)^{\tau},$ then
 by Lemma \ref{inducenorm} and  \ref{inner},
 \[T^{-1}\mathbf{\gamma}^{\tau}\tilde{\Psi}\{\mathbf{X}_{t,T}\}\tilde{\Psi}^{\tau}\{\mathbf{X}_{t,T}\}\mathbf{\gamma}=\parallel\mathbf{g}\parallel_{T}^{2}
\asymp\parallel\mathbf{g}\parallel^{2}\asymp|\mathbf{\gamma}|^2,\]
which implies $T^{-1}\sum_{t=1}^{T}\tilde{\Psi}\{\mathbf{X}_{t,T}\}\tilde{\Psi}^{\tau}\{\mathbf{X}_{t,T}\}$ has eigenvalues bounded away from 0 and $\infty.$
Therefore, $T^{-1}\sum_{t=1}^{T}\Psi_{k}\{X_{t,T}^{\left(k\right)}\}\Psi^{\tau}_{k}\{X_{t,T}^{\left(k\right)}\}$ also has eigenvalues bounded away from 0 and $\infty.$
 \end{proof}

\textbf{Proof for Proposition 1.}\\
Let $V_{sj}=\big\{1,\psi_{11}(X_{t_{sj},T}^{\left(1\right)}),\cdots,
\psi_{1J_{1,A}}(X_{t_{sj},T}^{\left(1\right)}),\cdots,\psi_{pJ_{p,A}}(X_{t_{sj},T}^{\left(p\right)})\big\}^{\tau},$ 
then
\[
\mathbf{\hat{\alpha}}_{s}=\left(\hat{C}_{0s},\hat{h}_{1}^{\left(s\right)\tau},\cdots,\hat{h}_{p}^{\left(s\right)\tau}\right)^{\tau}
=\Big\{\sum_{j=1}^{I_{T}}V_{sj}V_{sj}^{\tau}\Big\}^{-1}\sum_{j=1}^{I_{T}}V_{sj}Y_{t_{sj}}.
\]
Define $\omega_{s}\left(\mathbf{x}\right)=C_{0s}+\sum_{k=1}^{p}\beta_{k}^{\left(s\right)}\left(x_{k}\right),$
$\omega_{s,j}=\omega_s(\mathbf{X}_{t_{sj}})$
and \[\mathbf{\tilde{\alpha}}_{s}=\{\tilde{C}_{0s}, \tilde{h}_{1}^{\left(s\right)\tau},\cdots,\tilde{h}_{p}^{\left(s\right)\tau}\}^{\tau}
 =\Big\{\sum\limits_{j=1}^{I_{T}}V_{sj}V_{sj}^{\tau}\Big\}^{-1}\sum\limits_{j=1}^{I_{T}}V_{sj}\omega_{s,j}.\]
 Denote \[\tilde{\beta}_{k}^{\left(s\right)}\left(x_{k}\right)=\Psi_{k}\left(x_{k}\right)^{\tau}\tilde{h}_{k}^{\left(s\right)},\ \ \ \  \check{\beta}^{\left(s\right)}_{k}\left(x_{k}\right)=\Psi_{k}\left(x_{k}\right)^{\tau}\hat{h}_{k}^{\left(s\right)},\]
and $\tilde{\gamma}_{k}\left(x_{k}\right)=\frac{1}{N_{T}}\sum_{t=1}^{N_{T}}\tilde{\beta}_{k}^{\left(s\right)}\left(x_{k}\right),$
$\check{\gamma}_{k}\left(x_{k}\right)=\frac{1}{N_{T}}\sum_{s=1}^{N_{T}}\check{\beta}^{\left(s\right)}_{k}\left(x_{k}\right).$
Note that $\hat{\gamma}_{k}\left(x_{k}\right)=\check{\gamma}_{k}\left(x_{k}\right)-\check{\gamma}_{k}\left(0\right),$
by Cauchy-Schwartz  inequality,  we have
 \begin{equation*}
 \begin{aligned}
 \parallel\hat{\gamma}_{k}-\gamma_{k}\parallel^2_{L_{2}}
 \leq &4\parallel\check{\gamma}_{k}-\tilde{\gamma}_{k}\parallel^2_{L_{2}}
+4\parallel\tilde{\gamma}_{k}-\gamma_{k}\parallel^2_{L_{2}} \\
&+4|{\check\gamma}_{k}\left(0\right)-\tilde{\gamma}_{k}\left(0\right)|^2+4|\gamma_{k}\left(0\right)-\tilde{\gamma}_{k}\left(0\right)|^2.
 \end{aligned}
 \end{equation*}
 It suffices to deal with the approximation error terms
  $\parallel\tilde{\gamma}_{k}-\gamma_{k}\parallel^2_{L_{2}}$
 and stochastic error terms $\parallel\check{\gamma}_{k}-\tilde{\gamma}_{k}\parallel^2_{L_{2}}.$

\medskip

\textbf{ Approximate error terms:} We will show the rate of approximation error
\renewcommand{\theequation}{A.\arabic{equation}}
\begin{equation}\label{theorem11}
\parallel\tilde{\gamma}_{k}-\gamma_{k}\parallel^2_{L_{2}}=O\left(\rho_{A}^2/N_{T}\right),\ \ \
|\tilde{\gamma}_{k}\left(0\right)-\gamma_{k}\left(0\right)|^2=O\left(\rho_{A}^2/N_{T}\right).
\end{equation}
Note that $\gamma_{k}\left(x_{k}\right)=\sum_{s=1}^{N_{T}}\beta_{k}^{\left(s\right)}\left(x_{k}\right)/N_{T},$ we have
\begin{equation*}
\begin{aligned}
\parallel\tilde{\gamma}_{k}-\gamma_{k}\parallel^2_{L_{2}}=&\parallel N_{T}^{-1}\sum_{s=1}^{N_{T}}[
\tilde{\beta}_{k}^{\left(s\right)}\left(x_{k}\right)-\beta_{k}^{\left(s\right)}\left(x_{k}\right)]\parallel^2_{L_{2}}\\
&\preceq \frac{1}{N_{T}^2}\sum_{s=1}^{N_{T}}
\parallel\tilde{\beta}_{k}^{\left(s\right)}\left(x_{k}\right)-\beta_{k}^{\left(s\right)}\left(x_{k}\right)\parallel^2_{L_{2}}.
\end{aligned}
\end{equation*}
By the definition of $\rho_{A},$ there exists $\breve{h}_{k}^{\left(s\right)}=(\breve{h}_{kl}^{\left(s\right)},l=1,\cdots,J_{k,A})^{\tau}$
and $\breve{\beta}_{k}^{\left(s\right)}\left(x_{k}\right)=\Psi_{k}\left(x_{k}\right)^{\tau}\breve{h}_{k}^{\left(s\right)}$ such that
$\sup_{x_{k}\in[0,1]}\big|\breve{\beta}_{k}^{\left(s\right)}\left(x_{k}\right)-\beta_{k}^{\left(s\right)}\left(x_{k}\right)\big|=O\left(\rho_{A}\right).$
Therefore,
\begin{equation*}
\begin{aligned}
&\parallel\tilde{\gamma}_{k}-\gamma_{k}\parallel^2_{L_{2}}\\
\preceq &\frac{1}{N_{T}^2}\sum_{s=1}^{N_{T}}\parallel\tilde{\beta}_{k}^{\left(s\right)}\left(x_{k}\right)-\breve{\beta}_{k}^{\left(s\right)}\left(x_{k}\right)\parallel^2_{L_{2}}
+\frac{1}{N_{T}^2}\sum_{s=1}^{N_{T}}\parallel \breve{\beta}_{k}^{\left(s\right)}\left(x_{k}\right)-\beta_{k}^{\left(s\right)}\left(x_{k}\right)\parallel^2_{L_{2}}\\
=&\frac{1}{N_{T}^2}\sum_{s=1}^{N_{T}}\parallel\tilde{\beta}_{k}^{\left(s\right)}\left(x_{k}\right)-\breve{\beta}_{k}^{\left(s\right)}\left(x_{k}\right)\parallel^2_{L_{2}}
+O\left(\frac{\rho_{A}^2}{N_{T}}\right).
\end{aligned}
\end{equation*}
Let $\mathbf{\breve{\alpha}}^{\left(s\right)}=(C_{0s},\breve{h}_{1}^{\left(s\right)},\cdots,\breve{h}_{p}^{\left(s\right)})^{\tau},$ then
\begin{equation*}
\parallel\tilde{\beta}_{k}^{\left(s\right)}\left(x_{k}\right)-\breve{\beta}_{k}^{\left(s\right)}\left(x_{k}\right)\parallel^2_{L_{2}}
=\parallel\Psi_{k}\left(x_{k}\right)^T(\tilde{h}_{k}^{\left(s\right)}-\breve{h}_{k}^{\left(s\right)})\parallel^2_{L_{2}}
\asymp|\tilde{h}_{k}^{\left(s\right)}-\breve{h}_{k}^{\left(s\right)}|^{2}
\leq|\mathbf{\tilde{\alpha}}^{\left(s\right)}-\mathbf{\breve{\alpha}}^{\left(s\right)}|^{2}.
\end{equation*}
On the one hand,
$I_{T}^{-1}\sum_{j=1}^{I_{T}}[V_{sj}^{\tau}(\mathbf{\tilde{\alpha}}^{\left(s\right)}-\mathbf{\breve{\alpha}}^{\left(s\right)})]^2
\asymp|\mathbf{\tilde{\alpha}}^{\left(s\right)}-\mathbf{\breve{\alpha}}^{\left(s\right)}|^2$ by Lemma \ref{eigenvalue}.
On the other hand, $\sum_{j=1}^{I_{T}}V_{sj}\left(\omega_{s,j}-V_{sj}^{\tau}\mathbf{\tilde{\alpha}}^{\left(s\right)}\right)=0$ ensures that
\begin{equation*}
\begin{aligned}
&\frac{1}{I_{T}}\sum_{j=1}^{I_{T}}[V_{sj}^{\tau}(\mathbf{\tilde{\alpha}}^{\left(s\right)}-\breve{\alpha}^{\left(s\right)})]^{2}
\leq \frac{1}{I_{T}}\sum_{j=1}^{I_{T}}(\omega_{s,j}-V_{sj}^{\tau}\mathbf{\breve{\alpha}}^{\left(s\right)})^2\\
=&\frac{1}{I_{T}}\sum_{j=1}^{I_{T}}\big\{\sum_{k=1}^{p}\big[\beta_{k}^{\left(s\right)}(X_{t_{sj},T}^{\left(k\right)})
-\breve{\beta}_{k}^{\left(s\right)}(X_{t_{sj},T}^{\left(k\right)})\big]\big\}^{2}=O\left(\rho_{A}^{2}\right).
\end{aligned}
\end{equation*}
Thus, $|\mathbf{\tilde{\alpha}}^{\left(s\right)}-\mathbf{\breve{\alpha}}^{\left(s\right)}|^2=O\left(\rho_{A}^{2}\right),$ which means
\[\frac{1}{N_{T}^2}\sum_{s=1}^{N_{T}}\parallel\tilde{\beta}_{k}^{\left(s\right)}\left(x_{k}\right)-\breve{\beta}^{\left(s\right)}\left(x_{k}\right)\parallel^{2}_{L_{2}}
=O\left(\rho_{A}^2/N_{T}\right).\]
\textbf{ Stochastic error terms: }  We next show the rate of stochastic errors:
\begin{equation}\label{theorem12}
\parallel\check{\gamma}_{k}-\tilde{\gamma}_{k}\parallel^{2}_{L_{2}}=O_{p}\left(K_{A}/T\right),\ \ \
|\check{\gamma}_{k}\left(0\right)-\tilde{\gamma}_{k}\left(0\right)|^{2}=O_{p}\left(K_{A}/T\right).
\end{equation}
It is easy to see that
\begin{equation*}
\begin{aligned}
\parallel\check{\gamma}_{k}\left(x_{k}\right)-\tilde{\gamma}_{k}\left(x_{k}\right)\parallel^2_{L_{2}}
\preceq&\frac{1}{N_{T}^2}\sum_{s=1}^{N_{T}}\parallel\check{\beta}_{k}^{\left(s\right)}\left(x_{k}\right)
-\tilde{\beta}_{k}^{\left(s\right)}\left(x_{k}\right)\parallel^2_{L_{2}}\\
\leq &\frac{1}{N_{T}^2}\sum_{s=1}^{N_{T}}|\hat{\alpha}^{\left(s\right)}-\tilde{\alpha}^{\left(s\right)}|^2\\
=&\frac{1}{N_{T}^2}\sum_{s=1}^{N_{T}}\big| \big\{\sum_{j=1}^{I_{T}}V_{sj}V_{sj}^{\tau}\big\}^{-1}
\sum_{j=1}^{I_{T}}V_{sj}(Y_{t_{sj},T}-\omega_{s,j})\big|^2\\
\preceq&\frac{1}{N_{T}^2I_{T}^2}\sum_{s=1}^{N_{T}}\big|\sum_{j=1}^{I_{T}}V_{sj}\sigma_{t_{sj}}\varepsilon_{t_{sj}}\big|^2.
\end{aligned}
\end{equation*}
Note that $V_{sj}^{\tau}V_{sj'}=1+\sum_{k=1}^{p}\sum_{l=1}^{J_{k,A}}\Psi_{kl}\big\{X_{t_{sj},T}^{\left(k\right)}\big\}\Psi_{kl}\big\{X_{t_{sj'},T}^{\left(k\right)}\big\}.$
Under Assumption (A5), we obtain that
\begin{equation*}
\begin{aligned}
E\Big|\sum_{j=1}^{I_{T}}V_{sj}\sigma_{t_{sj}}\varepsilon_{t_{sj}}\Big|^2\leq& C_{\sigma}
E\Big\{\sum_{j=1}^{I_{T}}V_{sj}^{\tau}V_{sj}\varepsilon_{t_{sj}}^{2}
+\sum_{\substack{j=1\\j\neq j'}}^{I_{T}}V_{sj}^{\tau}V_{sj'}\varepsilon_{t_{sj}}\varepsilon_{t_{sj'}}\Big\}\\
=& C_{\sigma}E\Big\{\sum_{j=1}^{I_{T}}\big[1+\sum_{k=1}^{p}\sum_{l=1}^{J_{k,A}}\psi_{kl}^2(X_{t_{sj},T}^{\left(k\right)})\big]
\varepsilon_{t_{sj}}^{2}\Big\}\\
&+C_{\sigma}E\Big\{\sum_{\substack{j=1\\j\neq j'}}^{I_{T}}\big[1+\sum_{k=1}^{p}\sum_{l=1}^{J_{k,A}}\psi_{kl}(X_{t_{sj},T}^{\left(k\right)})
\psi_{kl}(X_{t_{sj'},T}^{\left(k\right)})\big]\varepsilon_{t_{sj}}\varepsilon_{t_{sj'}}\Big\}.
\end{aligned}
\end{equation*}
Assumption (A3) makes the second term be zero, and the first term is bounded by
$I_{T}+\sum_{j=1}^{I_{T}}\sum_{k=1}^{p}\sum_{l=1}^{J_{k,A}}E[\psi_{kl}(X_{t_{sj},T}^{\left(k\right)})]^{2}.$
However,
\begin{equation*}
\begin{aligned}
\sum_{k=1}^{p}\sum_{l=1}^{J_{k,A}}E[\psi_{kl}(X_{t_{sj},T}^{\left(k\right)})]^{2}
\preceq & \sum_{k=1}^{p}\sum_{l=1}^{J_{k,A}}E\big[\psi_{kl}(X_{t_{sj},T}^{\left(k\right)})
-\psi_{kl}(X_{t_{sj}}^{\left(k\right)}\left(t_{sj}/T\right))\big]^{2}\\
&+\sum_{k=1}^{p}\sum_{l=1}^{J_{k,A}}E[\psi_{kl}(X_{t_{sj}}^{\left(k\right)}\left(t_{sj}/T\right))]^{2}.
\end{aligned}
\end{equation*}
By Assumption (A2) and the properties of B-spline,
\begin{equation*}
\sum_{k=1}^{p}\sum_{l=1}^{J_{k,A}}E[\psi_{kl}(X_{t}^{\left(k\right)}\left(t/T\right))]^2
\asymp
\sum_{k=1}^{p}J_{k,A}=O\left(K_{A}\right).
\end{equation*}
On the other hand,
\begin{equation*}
\begin{aligned}
&E[\psi_{kl}(X_{t,T}^{\left(k\right)})-\psi_{kl}(X_{t}^{\left(k\right)}\left(t/T\right))]^{2}
=J_{k,A}E[B_{kl,A}(X_{t,T}^{\left(k\right)})-B_{kl,A}(X_{t}^{\left(k\right)}\left(t/T\right))]^{2}\\
\leq& CJ_{k,A}E[|\mathbf{X}_{t,T}-\mathbf{X}_{t}\left(t/T\right)|^2]
\leq CJ_{k,A}\frac{1}{T^{2}}E[U_{t,T}^{2}\left(t/T\right)].
\end{aligned}
\end{equation*}
Therefore,
\begin{equation*}
\begin{aligned}
\sum_{j=1}^{I_{T}}\sum_{k=1}^{p}\sum_{l=1}^{J_{k,A}}E[\psi_{kl}(X_{t_{sj},T}^{\left(k\right)})]^{2}
=O\left(K_{A}I_{T}\right)+\frac{I_{T}}{T^{2}}\sum_{k=1}^{p}J_{k,A}^{2}.
\end{aligned}
\end{equation*}
and in turn
\begin{equation*}
\begin{aligned}
\frac{1}{T^{2}}\sum_{s=1}^{N_{T}}\big|\sum_{j=1}^{I_{T}}V_{sj}\sigma_{t_{sj}}\varepsilon_{t_{sj}}\big|^2
=O_p\left(\frac{K_{A}}{T}+\frac{K_{A}^{2}}{T^{3}}\right)
=O_p\left(\frac{K_{A}}{T}\right),
\end{aligned}
\end{equation*}
which completes the proof of \eqref{theorem12} and hence the first half of Theorem 1. The rest is direct from Lemma \ref{eigenvalue}.
\qed\\

 \textbf{Proof for Theorem \ref{alpha}.} Let
$m\left(u,\mathbf{x}\right)=\alpha_{0}\left(u\right)+\sum_{k=1}^{p}\delta_{k}\left(u\right)\gamma_{k}\left(x_{k}\right),$
$\hat{m}\left(u,\mathbf{x}\right)=\alpha_{0}\left(u\right)+\sum_{k=1}^{p}\delta_{k}\left(u\right)\hat{\gamma}_{k}\left(x_{k}\right),$
where $\mathbf{x}=\left(x_{1},\cdots,x_{p}\right)^{\tau}.$ Denote
$m_{t}=m\left(\frac{t}{T},\mathbf{X}_{t,T}\right),$  $\hat{m}_{t}=\hat{m}\left(\frac{t}{T},\mathbf{X}_{t,T}\right).$

Denote $\widetilde{D}_{t}=\{\hat{\Gamma}\left(\mathbf{X}_{t,T}\right)^{\tau}\mathbf{\Phi}\left(t/T\right)\}^{\tau},$ where
$\mathbf{\Phi}\left(\cdot\right)=diag\left( \Phi_{0}\left(\cdot\right)^{\tau},\cdots,\Phi_{p}\left(\cdot\right)^{\tau}\right),$
 $\Phi_{k}(\cdot)=\{\varphi_{k1}(\cdot),\cdots,\varphi_{kJ_{k,C}}(\cdot)\}^{\tau},$ and
 $\hat{\Gamma}(\mathbf{X}_{t,T})=\{1,\hat{\gamma}_{1}(X_{t,T}^{\left(1\right)}),\cdots,\hat{\gamma}_{p}
(X_{t,T}^{\left(p\right)})\}^\tau.$
 Then $\mathbf{\hat{g}}=\big\{\sum\limits_{t=1}^{T}\widetilde{D}_{t}\widetilde{D}_{t}^{\tau}\big\}^{-1}\sum\limits_{t=1}^{T}\widetilde{D}_{t}Y_{t,T}.$

Furthermore, assuming that $\mathbf{\tilde{g}}=\left(\tilde{g}_{0}^{\tau},\cdots,\tilde{g}_{p}^{\tau}\right)^{\tau}$
with $\tilde{g}_{k}=\left(g_{k1},\cdots,g_{kJ_{k,A}}\right)^{\tau}$
is given by
\[\mathbf{\tilde{g}}=\big\{\sum_{t=1}^{T}\widetilde{D}_{t}\widetilde{D}_{t}^{\tau}\big\}^{-1}\sum_{t=1}^{T}\widetilde{D}_{t}m_{t}\]
and $\tilde{\alpha}_{k}\left(u\right)=\Phi_{k}\left(u\right)^{\tau}\tilde{g}_{k}$ for $k=0,\cdots,p.$

By Cauchy-Schwartz  inequality and identifiable condition $\parallel\alpha_{k}\parallel_{L_{2}}=1,$ we get
\begin{equation*}
\begin{aligned}
\parallel \hat{\alpha}_{k}-\alpha_{k}\parallel^{2}_{L_{2}}=&\parallel \hat{\delta}_{k}/\parallel\hat{\delta}_{k}\parallel_{L_{2}}-\alpha_{k}\parallel^2_{L_{2}}\\
\leq& 2(1-\parallel\hat{\delta}_{k}\parallel_{L_{2}})^{2}+2\parallel\hat{\delta}_{k}-\alpha_{k}\parallel^{2}_{L_{2}}\\
\leq&4\parallel\hat{\delta}_{k}-\alpha_{k}\parallel^{2}_{L_{2}}\\
\leq&4 \parallel\hat{\delta}_{k}-\tilde{\alpha}_{k}\parallel^{2}_{L_{2}}
+4\parallel\tilde{\alpha}_{k}-\alpha_{k}\parallel^{2}_{L_{2}}.
\end{aligned}
\end{equation*}
\textbf{ Approximate error term: } We  show the rate of approximate error term as follows
\begin{equation}\label{theorem21}
\parallel\tilde{\alpha}_{k}-\alpha_{k}\parallel^{2}_{L_{2}}=O\left(\frac{\rho_{A}^{2}}{N_{T}}+\frac{K_{A}}{T}+\rho_{C}^{2}\right).
\end{equation}
By the definition of $\rho_{C},$ there exists $\mathbf{\breve{g}}=\left(\breve{g}_{0},\cdots,\breve{g}_{p}\right)$ such that $\breve{\alpha}_{k}\left(u\right)=\Phi_{k}\left(u\right)^{\tau}\breve{g}_{k}$
satisfying \[\sup_{u\in[0,1]}|\breve{\alpha}_{k}\left(u\right)-\alpha_{k}\left(u\right)|=O\left(\rho_{C}\right)\]
for $k=0,\cdots,p.$
Thus $\parallel\breve{\alpha}_{k}-\alpha_{k}\parallel^{2}_{L_{2}}=O\left(\rho_{C}^{2}\right).$

Note that $\parallel\tilde{\alpha}_{k}-\breve{\alpha}_{k}\parallel^{2}_{L_{2}}\asymp |\mathbf{\tilde{g}}-\mathbf{\breve{g}}|^{2}$
and the normal equation $\sum_{t=1}^{T}\widetilde{D}_{t}\big\{m_{t}-\widetilde{D}_{t}^{\tau}\mathbf{\tilde{g}}\big\}=0$ yields
\begin{equation*}
\begin{aligned}
|\mathbf{\tilde{g}}-\mathbf{\breve{g}}|^{2}\asymp &
\frac{1}{T}\sum_{t=1}^{T}|\widetilde{D}_{t}^{\tau}\left(\mathbf{\tilde{g}}-\mathbf{\breve{g}}\right)|^{2}
\asymp&\frac{1}{T}\sum_{t=1}^{T}| m_{t}-\hat{m}_{t}|^{2}
+\frac{1}{T}\sum_{t=1}^{T}|\hat{m}_{t}-\widetilde{D}_{t}^{\tau}\mathbf{\breve{g}}|^{2}.
\end{aligned}
\end{equation*}

According to  Proposition 1 and boundness of $\delta_{k},$
\begin{equation*}
\begin{aligned}
\frac{1}{T}\sum_{t=1}^{T}|m_{t}-\hat{m}_{t}|^{2}
\preceq \frac{1}{T}\sum_{t=1}^{T}\sum_{k=1}^{p}[\gamma_{k}(X_{t,T}^{\left(k\right)})
-\hat{\gamma}_{k}(X_{t,T}^{\left(k\right)})]^{2}
=O_{p}\left(\frac{\rho_{A}^{2}}{N_{T}}+\frac{K_{A}}{T}\right).
\end{aligned}
\end{equation*}
 On the other hand, from Assumption (A7)
\begin{equation*}
\begin{aligned}
&\frac{1}{T}\sum_{t=1}^{T}|\hat{m}_{t}-\widetilde{D}_{t}^{\tau}\mathbf{\breve{g}}|^{2}\\ \preceq&
\frac{1}{T}\sum_{t=1}^{T}[\alpha_{0}\left(t/T\right)-\breve{\alpha}_{0}\left(t/T\right)]^{2}
+\frac{1}{T}\sum_{t=1}^{T}\sum_{k=1}^{p}\hat{\gamma}_{k}^{2}(X_{t,T}^{\left(k\right)})
[\delta_{k}\left(t/T\right)-\breve{\delta}_{k}\left(t/T\right)]^{2}\\
\preceq&\rho_{C}^{2}+\rho_{C}^{2}\frac{1}{T}\sum_{k=1}^{p}\sum_{t=1}^{T}\hat{\gamma}_{k}^{2}(X_{t,T}^{\left(k\right)})\\
=&O\left(\rho_{C}^{2}\right).
\end{aligned}
\end{equation*}

\textbf {Stochastic Error terms:} We next show the following rate of stochastic error term
\begin{equation}\label{theorem22}
\parallel\hat{\delta}_{k}-\tilde{\alpha}_{k}\parallel^{2}_{L_{2}}=O_p\left(K_{C}/T\right).
\end{equation}
It is easy to see \[\parallel\hat{\delta}_{k}-\tilde{\alpha}_{k}\parallel^{2}_{L_{2}}=
\parallel\Phi_{k}\left(u\right)^{\tau}\left(\hat{g}_{k}-\tilde{g}_{k}\right)\parallel^{2}_{L_{2}}\asymp
|\hat{g}_{k}-\tilde{g}_{k}|^{2}\leq|\mathbf{\hat{g}}-\mathbf{\tilde{g}}|^{2},\]
and
\begin{equation*}
\begin{aligned}
\mathbf{\hat{g}}-\mathbf{\tilde{g}}=&\big\{\sum_{t=1}^{T}\widetilde{D}_{t}\widetilde{D}_{t}^{\tau}\big\}^{-1}
\sum_{t=1}^{T}\widetilde{D}_{t}\left(Y_{t}-m_{t}\right)\\
=&\big\{\sum_{t=1}^{T}\widetilde{D}_{t}\widetilde{D}_{t}^{\tau}\big\}^{-1}
\sum_{t=1}^{T}\widetilde{D}_{t}\sigma(t/T,\mathbf{X}_{t,T})                                                                                                                                                                                                                     \varepsilon_{t}.
\end{aligned}
\end{equation*}
Based on  Assumption (A5), it is sufficient to bound
$E|T^{-1}\sum_{t=1}^{T}\widetilde{D}_{t}\varepsilon_{t}|^{2}.$

Let $\Gamma\left(\mathbf{X}_{t,T}\right)=\{1,\gamma_{1}\left(X_{1t}\right),\cdots,
\gamma_{p}\left(X_{pt}\right)\}^{\tau}$ and
$D_{t}=\{\Gamma\left(\mathbf{X}_{t,T}\right)^{\tau}\mathbf{\Phi}\left(t/T\right)\}^{\tau},$
then \[|T^{-1}\sum_{t=1}^{T}\widetilde{D}_{t}\varepsilon_{t}|^{2}
\leq2|T^{-1}\sum_{t=1}^{T}(\widetilde{D}_{t}-D_{t})\varepsilon_{t}|^{2}+2|T^{-1}\sum_{t=1}^{T}D_{t}\varepsilon_{t}|^{2}.\]
Obviously, under Assumption (A3),
$E\big|T^{-1}\sum_{t=1}^{T}D_{t}\varepsilon_{t}\big|^{2}
=\frac{1}{T^{2}}\sum_{t=1}^{T}E[D_{t}^{\tau}D_{t}].
$
From Assumption (A1) and Assumption (A7),
\begin{equation*}
\begin{aligned}
E[\gamma_{k}(X_{t,T}^{\left(k\right)})]^{2}
\preceq& E[\gamma_{k}(X_{t,T}^{\left(k\right)})-\gamma_{k}(X_{t}^{\left(k\right)}(t/T))]^{2}
+E[\gamma_{k}(X_{t}^{\left(k\right)}(t/T))]^{2}\\
\preceq & E|\mathbf{X}_{t,T}-\mathbf{X}_{t}\left(t/T\right)|^{2}+O\left(1\right)\\
=&O\left(T^{-2}\right)+O\left(1\right)=O\left(1\right),
\end{aligned}
\end{equation*}
where $X_{t}^{\left(k\right)}\left(u\right)$ is the $k$-th component of stationary approximation process $\{\mathbf{X}_{t}\left(u\right)\}$ of
locally stationary process $\{\mathbf{X}_{t,T}\}$ at rescaled time $u.$
In combination with $D_{t}^{\tau}D_{t}=\sum_{l=1}^{J_{0,C}}\varphi_{0l}^{2}(t/T)
+\sum_{k=1}^{p}\sum_{l=1}^{J_{k,C}}\varphi_{kl}^{2}(t/T)\gamma^{2}_{k}(X_{t,T}^{\left(k\right)}),$ we have
\begin{equation*}
\begin{aligned}
E[D_{t}^{\tau}D_{t}]\preceq &\sum_{l=1}^{J_{0,C}}\varphi_{0l}^{2}\left(t/T\right)+\sum_{k=1}^{p}\sum_{l=1}^{J_{k,C}}\varphi_{kl}^{2}\left(t/T\right)
=O\left(K_{C}\right),
\end{aligned}
\end{equation*}
which means $|T^{-1}\sum_{t=1}^{T}D_{t}\varepsilon_{t}|^{2}=O_p\left(K_{C}/T\right).$
Meantime,
\begin{equation*}
\begin{aligned}
&E|T^{-1}(\widetilde{D}_{t}-D_{t})\varepsilon_{t}|^{2}
=&\frac{1}{T^{2}}\sum_{t=1}^{T}\sum_{k=1}^{p}\sum_{l=1}^{J_{k,C}}\varphi_{kl}^{2}\left(t/T\right)
E[\hat{\gamma}_{k}(X_{t,T}^{\left(k\right)})-\gamma_{k}(X_{t,T}^{\left(k\right)})]^{2}.
\end{aligned}
\end{equation*}
By Cauchy-Schwartz inequality,
\begin{equation*}
\begin{aligned}
&E[\hat{\gamma}_{k}(X_{t,T}^{\left(k\right)})-\gamma_{k}(X_{t,T}^{\left(k\right)})]^{2}
\leq 3 E[\hat{\gamma}_{k}(X_{t,T}^{\left(k\right)})-\hat{\gamma}_{k}(X_{t}^{\left(k\right)}\left(t/T\right))]^{2}\\
&+3E[\hat{\gamma}_{k}(X_{t}^{\left(k\right)}\left(t/T\right))-\gamma_{k}(X_{t}^{\left(k\right)}\left(t/T\right))]^{2}
+3E[\gamma_{k}(X_{t}^{\left(k\right)}\left(t/T\right))-\gamma_{k}\big(X_{t,T}^{\left(k\right)})]^{2}.
\end{aligned}
\end{equation*}
The third term is bounded by $O\left(T^{-2}\right)$ from Assumption (A1) and (A7).
 Assumption (A2) and Proposition 1 ensure that the second term is bounded by
 $O_{p}\left(\rho_{A}^{2}/N_{T}+K_{A}/T\right).$ For the first term,  we note that
\begin{equation*}
\begin{aligned}
&E[\hat{\gamma}_{k}(X_{t,T}^{\left(k\right)})-\hat{\gamma}_{k}(X_{t}^{\left(k\right)}\left(t/T\right))]^{2}\\
=&E\big[N_{T}^{-1}\sum_{s=1}^{N_{T}}\sum_{l=1}^{J_{k,A}}\hat{h}^{\left(s\right)}_{kl}J_{k,A}^{1/2}
\{B_{kl,A}(X_{t,T}^{\left(k\right)})-B_{kl,A}(X_{t}^{\left(k\right)}\left(t/T\right))\}\big]^{2}\\
\preceq &J_{k,A}N_{T}^{-2}\sum_{s=1}^{N_{T}}\sum_{l=1}^{J_{k,A}}(\hat{h}^{\left(s\right)}_{kl})^{2}
|\mathbf{X}_{t,T}-\mathbf{X}_{t}\left(t/T\right)|^{2}\\
=&O_p\left(\frac{K_{A}^{2}}{T^{2}N_{T}}\right),
\end{aligned}
\end{equation*}
and thus
\begin{equation*}
\begin{aligned}
&E|T^{-1}\sum_{t=1}^{T}(\widetilde{D}_{t}-D_{t})\varepsilon_{t}|^{2}=\frac{1}{T^{2}}
\sum_{t=1}^{T}\sum_{k=1}^{p}\sum_{l=1}^{J_{k,C}}\varphi_{kl}^{2}\left(t/T\right)
O_{p}\left(\frac{\rho_{A}^{2}}{N_{T}}+\frac{K_{A}}{T}\right)\\
\leq &\frac{1}{T^{2}}\sum_{t=1}^{T}\sum_{k=1}^{p}\big[\sum_{l=1}^{J_{k,C}}\varphi_{kl}\left(t/T\right)\big]^{2}
O_{p}\left(\frac{\rho_{A}^{2}}{N_{T}}+\frac{K_{A}}{T}\right)
=O_p\left(\frac{K_{C}}{T}\Big\{\frac{\rho_{A}^{2}}{N_{T}}+\frac{K_{A}}{T}\Big\}\right),
\end{aligned}
\end{equation*}
which shows \eqref{theorem22}.
\qed\\

\textbf{Proof for Theorem \ref{beta}.}
Let $\omega\left(u,\mathbf{x}\right)=\alpha_{0}\left(u\right)+\sum_{k=1}^{p}\alpha_{k}\left(u\right)\beta_{k}\left(x_{k}\right),$
$\omega_{t}=\omega\left(t/T,\mathbf{X}_{t,T}\right),$ and
$\hat{\omega}\left(u,\mathbf{x}\right)=\hat{\alpha}_{0}\left(u\right)+\sum_{k=1}^{p}\hat{\alpha}_{k}\left(u\right)\beta_{k}\left(x_{k}\right),$
$\hat{\omega}_{t}=\hat{\omega}\left(t/T,\mathbf{X}_{t,T}\right).$
Define $\widetilde{Z}_{t}=\{\hat{\Delta}\left(t/T\right)^{\tau}\mathbf{\Psi}\left(\mathbf{X}_{t,T}\right)\}^{\tau},$ where
$\hat{\Delta}\left(\cdot\right)=\{\hat{\alpha}_{1}\left(\cdot\right),\cdots,\hat{\alpha}_{p}\left(\cdot\right)\}^{\tau},$
$\mathbf{\Psi}\left(\cdot\right)=diag\left(\Psi_{1}\left(\cdot\right)^{\tau},\cdots,\Psi_{p}\left(\cdot\right)^{\tau}\right)$
and
$\Psi_{k}\left(\cdot\right)=\{\psi_{k1}\left(\cdot\right),\cdots,\psi_{kJ_{k,A}}\left(\cdot\right)\}^{\tau}$ for $ k=1,\cdots,p.$
Suppose $\mathbf{f}=\left(f_{1}^{\tau},\cdots,f_{p}^{\tau}\right)^{\tau}$ is given by
\[\mathbf{\tilde{f}}=\big\{\sum_{t=1}^{T}\widetilde{Z}_{t}\widetilde{Z}_{t}^{\tau}\big\}^{-1}\sum_{t=1}^{T}\widetilde{Z}_{t}\omega_{t},\]
and $\tilde{\beta}_{k}\left(x_{k}\right)=\Psi_{k}\left(x_{k}\right)^{\tau}\tilde{f}_{k}.$
Analogously,  represent  $Z_{t}=\{\Delta\left(t/T\right)^{\tau}\mathbf{\Psi}\left(\mathbf{X}_{t,T}\right)\}^{\tau}$ with
$\Delta\left(\cdot\right)=\{\alpha_{1}\left(\cdot\right),\cdots,\alpha_{p}\left(\cdot\right)\}^{\tau}.$
Obviously,
\begin{equation*}
\begin{aligned}
\parallel\hat{\beta}_{k}-\beta_{k}\parallel^{2}_{L_{2}}\leq &
4\parallel\check{\beta}_{k}-\tilde{\beta}_{k}\parallel^{2}_{L_{2}}
+4\parallel\tilde{\beta}_{k}-\beta_{k}\parallel^{2}_{L_{2}}
+4|\check{\beta}_{k}\left(0\right)-\tilde{\beta}_{k}\left(0\right)|^{2}\\
&+4|\tilde{\beta}_{k}\left(0\right)-\beta_{k}\left(0\right)|^{2}.
\end{aligned}
\end{equation*}
{\bf Approximation Error Term: } The rate of approximation error term is given by
\begin{equation}\label{theorem1}
\begin{aligned}
\parallel\tilde{\beta}_{k}-\beta_{k}\parallel^{2}_{L_{2}}
=&O_p\left(\rho_{A}^{2}+\rho_{C}^{2}+\frac{K_{A}\vee K_{C} }{T}\right),\\
|\tilde{\beta}_{k}\left(0\right)-\beta_{k}\left(0\right)|^{2}
=&O_p\left(\rho_{A}^{2}+\rho_{C}^{2}+\frac{K_{A}\vee K_{C} }{T}\right).
\end{aligned}
\end{equation}
On the one hand, by the definition of $\rho_{A},$ there exists $\mathbf{f}^*=\left(f^*_{1},\cdots,f^*_{p}\right)^{\tau}$ and $\beta^*_{k}\left(x_{k}\right)=\Psi_{k}\left(x_{k}\right)^{\tau}f^*_{k}$
such that \[\sup_{x_{k}\in[0,1]}|\beta^*\left(x_{k}\right)-\beta_{k}\left(x_{k}\right)|=O\left(\rho_{A}\right),\]
which means $\parallel\beta^*_{k}-\beta_{k}\parallel^{2}_{L_{2}}=O\left(\rho_{A}^{2}\right).$
On the other hand,
\begin{equation*}
\parallel\tilde{\beta}_{k}-\beta^{*}_{k}\parallel_{L_{2}}^{2}
=\parallel\Psi_{k}\left(x_{k}\right)^{\tau}(\tilde{f}_{k}-f^{*}_{k})\parallel^{2}_{L_{2}}
\asymp |\tilde{f}_{k}-f^*_{k}|^{2}\preceq |\mathbf{\tilde{f}}-\mathbf{f}^*|^{2}.
\end{equation*}
Furthermore,
\begin{equation*}
\begin{aligned}
|\mathbf{\tilde{f}}-\mathbf{f}^*|^{2}\asymp&\frac{1}{T}\sum_{t=1}^{T}|\widetilde{Z}_{t}^{\tau}(\mathbf{\tilde{f}}-\mathbf{f}^*)|^{2}\leq
\frac{1}{T}\sum_{t=1}^{T}|\omega_{t}-\widetilde{Z}_{t}^{\tau}\mathbf{f}^*|^{2}\\
\leq&\frac{2}{T}\sum_{t=1}^{T}|\omega_{t}-\hat{\omega}_{t}|^{2}
+\frac{2}{T}\sum_{t=1}^{T}|\hat{\omega}_{t}-\widetilde{Z}_{t}^{\tau}\mathbf{f}^*|^{2}
\end{aligned}
\end{equation*}
since $\sum_{t=1}^{T}\widetilde{Z}_{t}(\omega_{t}-\widetilde{Z}_{t}^{\tau}\tilde{f})=0.$
According to Theorem 1,
\begin{equation*}
\begin{aligned}
&\frac{1}{T}\sum_{t=1}^{T}|\omega_{t}-\hat{\omega}_{t}|^{2}\\ \preceq&
\frac{1}{T}\sum_{t=1}^{T}[\alpha_{0}\left(t/T\right)-\hat{\alpha}_{0}\left(t/T\right)]^{2}+
\frac{1}{T}\sum_{k=1}^{p}[\alpha_{k}\left(t/T\right)-\hat{\alpha}_{k}\left(t/T\right)]^{2}\beta_{k}^{2}(X_{t,T}^{\left(k\right)})\\
=&O_{p}\left(\frac{\rho_{A}^{2}}{N_{T}}+\rho_{C}^{2}+\frac{K_{A}\vee K_{C} }{T}\right).
\end{aligned}
\end{equation*}
Finally, we note that
\begin{equation*}
\begin{aligned}
&\frac{1}{T}\sum_{t=1}^{T}|\hat{\omega}_{t}-\widetilde{Z}_{t}^{\tau}\mathbf{f}^*|^{2}
=\frac{1}{T}\sum_{t=1}^{T}\big|\sum_{k=1}^{p} \hat{\alpha}_{k}\left(t/T\right)
[\beta_{k}(X_{t,T}^{\left(k\right)})-\beta^*_{k}(X_{t,T}^{\left(k\right)})]\big|^{2}\\
\preceq&\frac{1}{T}\sum_{t=1}^{T}\sum_{k=1}^{p} \hat{\alpha}_{k}^{2}\left(t/T\right)[\beta_{k}(X_{t,T}^{\left(k\right)})
-\beta^*_{k}(X_{t,T}^{\left(k\right)})]^{2}
=O\left(\rho_{A}^{2}\right)
\end{aligned}
\end{equation*}
since for each $k=1,\cdots,p,$
\begin{equation*}
\begin{aligned}
\frac{1}{T}\sum_{t=1}^{T}\hat{\alpha}_{k}^{2}\left(t/T\right)\preceq&
\frac{1}{T}\sum_{t=1}^{T}[\hat{\alpha}_{k}\left(t/T\right)-\alpha_{k}\left(t/T\right)]^{2}
+\frac{1}{T}\sum_{t=1}^{T}\alpha_{k}^{2}\left(t/T\right)\\
=&\int_{0}^{1}\alpha_{k}^{2}\left(u\right)\mathrm{d}u+o\left(1\right)=O\left(1\right).
\end{aligned}
\end{equation*}
Therefore,  \eqref{theorem1} holds.\\

\textbf {Stochastic Error Term: } We will show the rate of stochastic error term:
\begin{equation}\label{theorem32}
\begin{aligned}
\parallel\check{\beta}_{k}-\tilde{\beta}_{k}\parallel_{L_{2}}
=&O_p\left(\frac{K_{A}}{T}+\frac{K_{A}\rho_{A}^{2}}{TN_{T}}+\frac{K_{A}\rho_{C}^{2}}{T}+\frac{\left(K_{A}\vee K_{C}\right)^2}{T^{2}}\right),\\
|\check{\beta}_{k}\left(0\right)-\tilde{\beta}_{k}\left(0\right)|^{2}
=&O_p\left(\frac{K_{A}}{T}+\frac{K_{A}\rho_{A}^{2}}{TN_{T}}+\frac{K_{A}\rho_{C}^{2}}{T}+\frac{\left(K_{A}\vee K_{C}\right)^2}{T^{2}}\right).
\end{aligned}
\end{equation}
Firstly,
\[\parallel\check{\beta}_{k}-\tilde{\beta}_{k}\parallel^2_{L_{2}}=\parallel\Psi_{k}\left(x_{k}\right)^\tau(\hat{f}_{k}-\tilde{f}_{k})\parallel^2_{L_{2}}
\asymp |\hat{f}_{k}-\tilde{f}_{k}|^2\preceq|\mathbf{\hat{f}}-\mathbf{\tilde{f}}|^2.\]
However,
\begin{equation*}
\begin{aligned}
\mathbf{\hat{f}}-\mathbf{\tilde{f}}=&\big\{\sum_{t=1}^T\widetilde{Z}_{t}\widetilde{Z}_{t}^\tau\big\}^{-1}\sum_{t=1}^T\widetilde{Z}_{t}\left(Y_{t,T}-\omega_{t}\right)
=\big\{\sum_{t=1}^T\widetilde{Z}_{t}\widetilde{Z}_{t}^\tau\big\}^{-1}\sum_{t=1}^T\widetilde{Z}_{t}\sigma\left(t/T,\mathbf{X}_{t,T}\right)\varepsilon_{t},
\end{aligned}
\end{equation*}
which implies \[|\mathbf{\hat{f}}-\mathbf{\tilde{f}}|^2\asymp |T^{-1}\sum_{t=1}^T\widetilde{Z}_{t}\varepsilon_{t}|^2
\preceq \frac{1}{T^2}\sum_{t=1}^T|(\widetilde{Z}_{t}-Z_{t})\varepsilon_{t}|^{2}
+\frac{1}{T^2}\sum_{t=1}^{T}| Z_{t}\varepsilon_{t}|^{2}\]
because of  Assumption (A5). Similar the counterpart in the proof of Theorem 1, we have
\begin{equation*}
\begin{aligned}
E|(\widetilde{Z}_{t}-Z_{t})\varepsilon_t|^2=&E[(\widetilde{Z}_{t}-Z_{t})^\tau(\widetilde{Z}_{t}-Z_{t})]\\
=&\sum_{k=1}^p[\hat{\alpha}_{k}\left(t/T\right)-\alpha_{k}\left(t/T\right)]^{2}\sum_{l=1}^{J_{k,A}}E[\psi_{kl}(X_{t,T}^{\left(k\right)})]^2.
\end{aligned}
\end{equation*}
Furthermore,
\begin{equation*}
\begin{aligned}
\sum_{l=1}^{J_{k,A}}E[\psi_{kl}(X_{t,T}^{\left(k\right)})]^2\preceq \sum_{l=1}^{J_{k,A}}E[\psi_{kl}(X_{t,T}^{\left(k\right)})
-\psi_{kl}(X_{t}^{\left(k\right)}\left(t/T\right))]^{2}
+\sum_{l=1}^{J_{k,A}}E[\psi_{kl}(X_{t}^{\left(k\right)}\left(t/T\right))]^2.
\end{aligned}
\end{equation*}
By Assumption (A1),
\begin{equation*}
\begin{aligned}
E[\psi_{kl}(X_{t,T}^{\left(k\right)})-\psi_{kl}(X_{t}^{\left(k\right)}\left(t/T\right))]^{2}
=&J_{k,A}E[B_{kl,A}(X_{t,T}^{\left(k\right)})-B_{kl,A}(X_{t}^{\left(k\right)}\left(t/T\right))]^{2}\\
\preceq & J_{k,A}E|X_{t,T}^{\left(k\right)}-X_{t}^{\left(k\right)}\left(t/T\right)|^{2}=O\left(J_{k,A}/T^{2}\right)
\end{aligned}
\end{equation*}
Assumption (A2) leads to
\begin{equation*}
\begin{aligned}
\sum_{l=1}^{J_{k,A}}E[\psi_{kl}(X_{t}^{\left(k\right)}\left(t/T\right))]^2
\asymp \sum_{l=1}^{J_{k,A}}\int\psi_{kl}^2\left(z\right)\mathrm{d}z
\leq \int\big[\sum_{l=1}^{J_{k,A}}\psi_{kl}\left(z\right)\big]^{2}\mathrm{d}z
=J_{k,A}.
\end{aligned}
\end{equation*}
Therefore, $\sum_{l=1}^{J_{k,A}}E[\psi_{kl}(X_{t,T}^{\left(k\right)})]^{2}
=O\left(J_{k,A}+J_{k,A}^{2}/T^{2}\right)=O\left(K_{A}+K_{A}^{2}/T^{2}\right),$ which yields
\begin{equation*}
\frac{1}{T^2}\sum_{t=1}^{T}E|(\widetilde{Z}_{t}-Z_{t})\varepsilon_{t}|^{2}
=O\left(\frac{K_{A}}{T}\Big\{\frac{\rho_{A}^{2}}{N_{T}}+\rho_{C}^{2}+\frac{K_{A}\vee K_{C}}{T}\Big\}\right).
\end{equation*}
Similarly,
\begin{equation*}
\begin{aligned}
\frac{1}{T^{2}}\sum_{t=1}^{T}E| Z_{t}\varepsilon_{t}|^{2}=&\frac{1}{T^{2}}\sum_{t=1}^{T}E[Z_{t}^{\tau}Z_{t}]\\
=&\frac{1}{T^{2}}\sum_{t=1}^{T}\Delta\left(t/T\right)^{\tau}
E[\mathbf{\Psi}\left(\mathbf{X}_{t,T}\right)\mathbf{\Psi}\left(\mathbf{X}_{t,T}\right)^{\tau}]\Delta\left(t/T\right).
\end{aligned}
\end{equation*}
Note that
\begin{equation*}
\begin{aligned}
E[\Psi_{k}\left(\mathbf{X}_{t,T}\right)^{\tau}\Psi_{k}\left(\mathbf{X}_{t,T}\right)]
=&E\big[\sum_{l=1}^{J_{k,A}}\psi_{kl}^{2}(X_{t,T}^{\left(k\right)})\big]\leq\int
\big\{\sum_{l=1}^{J_{k,A}}\psi_{kl}\left(z\right)\big\}^{2}f_{X_{t,T}}^{\left(k\right)}
\left(z\right)\mathrm{d}z\\=&J_{k,A},
\end{aligned}
\end{equation*}
where $f_{X_{t,T}}^{\left(k\right)}$ is the marginal density of $k$-th component of $\mathbf{X}_{t,T}.$
So, \[E[\mathbf{\Psi}\left(\mathbf{X}_{t,T}\right)\mathbf{\Psi}\left(\mathbf{X}_{t,T}\right)^{\tau}]=diag\left(J_{k,A}\right)_{k=1}^{p}=K_{A}I_{p}\]
and
\begin{equation*}
\begin{aligned}
\frac{1}{T^{2}}\sum_{t=1}^{T}E|Z_{t}\varepsilon_{t}|^{2}=
K_{A}\frac{1}{T^{2}}\sum_{t=1}^{T}\sum_{k=1}^{p}\alpha_{k}^{2}\left(\frac{t}{T}\right)=O\left(K_{A}/T\right),
\end{aligned}
\end{equation*}
which completes the proof of  \eqref{theorem32}.\qed\\

\textbf {Proof for Theorem \ref{Cadditive}:}
\begin{itemize}
\item[(i)]Without loss of generality, we assume the true model is
\begin{equation*}
\begin{aligned}
Y_{t,T}=&\alpha_{0}(t/T)+\sum_{k=1}^{p_{1}
+p_{2}}\alpha_{k}(t/T)\beta_{k}(X_{t,T}^{\left(k\right)})
+\sum_{k=p_{1}+p_{2}+1}^pc_{k}\beta_{k}(X_{t,T}^{\left(k\right)})\\
&+\sigma(t/T,\mathbf{X}_{t,T})\varepsilon_{t}.
\end{aligned}
\end{equation*}
Let
\begin{equation*}
\begin{aligned}
m_0\left(u\right)=&\alpha_{0}\left(u\right)+\sum_{k=1}^{p_{1}+p_{2}}\alpha_{k}\left(u\right)\hat{\beta}_{k}(X_{t,T}^{\left(k\right)})
+\sum_{k=p_{1}+p_{2}+1}^{p}c_{k}\hat{\beta}_{k}(X_{t,T}^{\left(k\right)}),\\
\end{aligned}
\end{equation*}
and $\mathcal{M}_{T,0}$ as the collection of all functions having form
\begin{equation*}
\pi_{0}^{\tau}\Phi_{0}\left(u\right)+
\sum_{k=1}^{p_{1}+p_{2}}\pi_{k}^{\tau}\Phi_{k}\left(u\right)\hat{\beta}_{k}(X_{t,T}^{\left(k\right)})
+\sum_{k=p_{1}+p_{2}+1}^pc_{k}\hat{\beta}_{k}(X_{t,T}^{\left(k\right)}).
\end{equation*}

It is sufficient to show  $Q_{1}\left(m_{T,0}\right)\leq Q_{1}\left(m_{T,0}+g_{T,1}\right)$
for any $m_{T,0}\in\mathcal{M}_{T,0}$ such that $\parallel m_{T,0}-m_{0}\parallel_{L_{2}}=O\left(\theta_{T}\right)$
and for any \[g_{T,1}\left(u\right)=g\left(u\right)\hat{\beta}_{p_{1}+p_{2}+1}(X_{t,T}^{\left(p_{1}+p_{2}+1\right)})\]
where $g\left(u\right)\in span\{\Phi_{p_{1}+p_{2}+1}\left(u\right)\}$ such that
$\parallel g\parallel_{L_{2}}\leq C\theta_{T}.$

For the sake of convenient presentation, we also
denote $Q_{1}\left(\mathbf{\pi}\right)$ as $Q_{1}\left(\mathbf{m}\right)$ if $\mathbf{m}=\mathbf{D}_{S}\mathbf{\pi}.$
Let \[\mathbf{m}_{T}= \{m_{T,0}\left(1/T\right),\cdots,m_{T,0}\left(T/T\right)\}^{\tau}\] 
and \[\mathbf{g}_{T}= \{g_{T,1}\left(1/T\right),\cdots,g_{T,1}\left(T/T\right)\}^{\tau},\]
then
\begin{equation*}
\begin{aligned}
&Q_{1}\left(\mathbf{m}_{T}\right)-Q_{1}\left(\mathbf{m}_{T}+\mathbf{g}_{T}\right)\\
=&\frac{1}{2}\left( \mathbf{Y}-\mathbf{m}_{T}\right)^{\tau}\left( \mathbf{Y}-\mathbf{m}_{T}\right)
-\frac{1}{2}\left( \mathbf{Y}-\mathbf{m}_{T}-\mathbf{g}_{T}\right)^{\tau}\left( \mathbf{Y}-\mathbf{m}_{T}-\mathbf{g}_{T}\right)\\
&-Tp_{\lambda_{T}}(K_{C}^{-3/2}
\parallel g'\parallel_{L_{2}}).
\end{aligned}
\end{equation*}

Furthermore,
\begin{equation*}
\begin{aligned}
&\left( \mathbf{Y}-\mathbf{m}_{T}\right)^{\tau}\left( \mathbf{Y}-\mathbf{m}_{T}\right)
-\left( \mathbf{Y}-\mathbf{m}_{T}-\mathbf{g}_{T}\right)^{\tau}\left( \mathbf{Y}-\mathbf{m}_{T}-\mathbf{g}_{T}\right)\\
=&\langle \mathbf{g}_{T}, 2\mathbf{Y}-2\mathbf{m}_{T}-\mathbf{g}_{T}\rangle\\
\leq&| \mathbf{g}_{T}|\cdot| 2\mathbf{Y}-2\mathbf{m}_{T}-\mathbf{g}_{T}|,
\end{aligned}
\end{equation*}
where $\langle a,b\rangle$ is the inner product of vector $a$ and $b.$

Let $\mathbf{m}_{0}=\left(m_{0}\left(1/T\right),\cdots,m_{0}\left(T/T\right)\right)^{\tau},$ we have
\begin{equation*}
\begin{aligned}
T^{-1}|2\mathbf{Y}-2\mathbf{m}_{T}-\mathbf{g}_{T}|^{2}
\preceq& T^{-1}|\mathbf{Y}-\mathbf{m}_{0}|^{2}+T^{-1}|\mathbf{m}_{0}-\mathbf{m}_{T}|^{2}+T^{-1}| \mathbf{g}_{T}|^{2}\\
\end{aligned}
\end{equation*}
and
\begin{equation*}
\begin{aligned}
T^{-1}|\mathbf{g}_{T}|^{2}=&\frac{1}{T}\sum_{t=1}^{T}g^{2}\left(t/T\right)\hat{\beta}^{2}_{p_{1}+p_{2}+1}(X_{t,T}^{\left(p_{1}+p_{2}+1\right)})
\preceq\parallel g\parallel_{L_{2}}^{2}+o\left(1\right),
\end{aligned}
\end{equation*}
where the last step holds since Assumption (A7) and Theorem \ref{beta}.

Therefore,
\begin{equation*}
\begin{aligned}
&Q_{1}\left(\mathbf{m}_{T}\right)-Q_{1}\left(\mathbf{m}_{T}+\mathbf{g}_{T}\right)\\
 \preceq&\frac{1}{2}T^2\lambda_{T}\parallel g'\parallel_{L_{2}}
\Big\{\frac{\parallel g\parallel_{L_{2}}}{\lambda_{T}\parallel g'\parallel_{L_{2}}}
\left(1+o\left(1\right)\right)-2K_{C}^{-3/2}\frac{p'_{\lambda_{T}}\left(z\right)}{\lambda_{T}T}\Big\},
\end{aligned}
\end{equation*}
in which $z$ lies between 0 and $K_{C}^{-3/2}\parallel g'\parallel_{L_{2}}$
 and $\parallel g\parallel_{L_{2}}=O\left(\theta_{T}\right).$
The proof of part(i) is completed since $d_{p_{1}+p_{2}+1}\ge 2,$
$\theta_{T}/\lambda_{T}=o_{p}\left(1\right)$ and
$\liminf_{T\to\infty}\liminf_{\theta\to 0+}p'_{\lambda_{T}}\left(\theta\right)/\lambda_{T}>0.$\qed

\item[(ii)]According to Theorem 6  (p149) of \cite{deboor1978practical}, under Assumption (A6), there exists
$\tilde{\pi}_{k} =\left(\tilde{\pi}_{k1},\cdots,\tilde{\pi}_{kJ_{k,C}}\right)^{\tau}$ and  $\tilde{\alpha}_{k}=\tilde{\pi}_{k}^{\tau}\Phi_{k}$ such that
$\parallel\alpha_{k}-\tilde{\alpha}_{k}\parallel_{L_{2}}=O\left(K_{k,C}^{-d_{k}}\right).$
Let $\mathbf{\tilde{\pi}}=\left(\tilde{\pi}_{0}^{\tau},\cdots,\tilde{\pi}_{p}^{\tau}\right)^{\tau}$
and $\mathbf{u}=\left(u_{0}^{\tau},\cdots,u_{p}^{\tau}\right)^{\tau}\in\mathbb{R}^{N_{p}}$ with $N_{p}=\sum_{k=0}^{p}J_{k,C}.$
Next, we will show that for any given $\varepsilon>0,$
there is a sufficiently large $C$ such that
\begin{equation}\label{ball}
P\left(\inf_{|\mathbf{u}|=C}Q_{1}\left(\tilde{\mathbf{\pi}}
+\theta_{T}\mathbf{u}\right)>Q_{1}\left(\tilde{\mathbf{\pi}}\right)\right)\ge 1-\varepsilon.
\end{equation}

According to Lemma 3 of  \cite{Huang2010}, $W_{k}$ has eigenvalues bounded away from 0 and $\infty$ with probability tending to one
as $T\to\infty.$ Therefore,
\begin{equation*}
\begin{aligned}
&D_{T}\left(\mathbf{u}\right)=Q_{1}\left(\mathbf{\tilde{\pi}}+\theta_{T}\mathbf{u}\right)-Q_{1}\left(\mathbf{\tilde{\pi}}\right)\\
\ge &\frac{1}{2}\theta_{T}^{2}\mathbf{u}^{\tau}\mathbf{D}_{S}^{\tau}\mathbf{D}_{S}\mathbf{u}-\theta_{T}\mathbf{u}^{\tau}\mathbf{D}_{S}^{\tau}
\left(Y-\mathbf{D}_{S}\tilde{\mathbf{\pi}}\right)
-cT\lambda_{T}\sum_{k=1}^{p}|\theta_{T}u_{k}|,\\
\ge &\frac{1}{2}\theta_{T}^{2}\mathbf{u}^{\tau}\mathbf{D}_{S}^{\tau}\mathbf{D}_{S}\mathbf{u}-\theta_{T}\mathbf{u}^{\tau}\mathbf{D}_{S}^{\tau}
\left(Y-\mathbf{D}_{S}\tilde{\mathbf{\pi}}\right)-cT\lambda_{T}\theta_{T}\sqrt{p}|\mathbf{u}|,\\
\end{aligned}
\end{equation*}
where we use the fact that $p_{\lambda}\left(0\right)=0$ and $|p_{\lambda}\left(s\right)-p_{\lambda}\left(t\right)|\leq \lambda\left(s-t\right)$ for $s,t>0.$
Notice that the first term
$\frac{1}{2}\theta_{T}^{2}\mathbf{u}^{\tau}\mathbf{D}_{S}^{\tau}\mathbf{D}_{S}\mathbf{u}=O_{p}\left(T\theta_{T}^{2}|\mathbf{u}|^{2}\right).$
We also may choose  the sufficiently large $C$ such that the third term can be dominated by the first term uniformly on $|\mathbf{u}|=C.$
Finally, we observe that the $t$-th element of $\mathbf{Y}-\mathbf{D}_{S}\tilde{\mathbf{\pi}}$ is given by
\begin{equation*}
\begin{aligned}
\left(\mathbf{Y}-\mathbf{D}_{S}\mathbf{\tilde{\pi}}\right)_{t}
=&\alpha_{0}\left(t/T\right)-\tilde{\alpha}_{0}\left(t/T\right)+\sum_{k=1}^{p}\alpha_{k}\left(t/T\right)\beta_{k}(X_{t,T}^{\left(k\right)})\\
&-\sum_{k=1}^{p}\tilde{\alpha}_{k}\left(t/T\right)\hat{\beta}_{k}(X_{t,T}^{\left(k\right)})+\sigma\left(t/T,\mathbf{X}_{t,T}\right)\varepsilon_{t},
\end{aligned}
\end{equation*}
which is bounded by $O_p\left(1\right).$ Thus, the second is bounded by $T\theta_{T}|\mathbf{u}|,$ which is also dominated by the first term.
In combination with the nonnegativity of the first term, we show \eqref{ball}, which implies with probability at least $1-\varepsilon$ that
there exists a local minimizer in the ball
$\{\tilde{\mathbf{\pi}}+\theta_{T}\mathbf{u}:|\mathbf{u}|\leq C\},$  i.e., $|\hat{\mathbf{\pi}}-\tilde{\pi}|^2=O_{p}\left(\theta_{T}^{2}\right).$
Again by the property of B-spline, we have that
$\parallel\hat{\mathbf{\pi}}^{\tau}\Phi_{k}-\tilde{\mathbf{\pi}}^{\tau}\Phi_{k}\parallel_{L_{2}}^2=O\left(|\hat{\mathbf{\pi}}-\tilde{\mathbf{\pi}}|^{2}\right)
=O_{p}\left(\theta_{T}^{2}\right).$ The proof is finished  in combination with
 $\parallel\alpha_{k}-\tilde{\alpha}_{k}\parallel_{L_{2}}=O\left(K_{k,C}^{-d_{k}}\right).$\qed
\end{itemize}
The proofs for Theorem \ref{Cvarying}   is very similar to Theorem \ref{Cadditive},  and thus omitted here.



\end{document}